\documentclass[12pt,a4paper]{article}
\usepackage{amssymb,times1,amsmath,amsfonts,amsthm,euscript}
\usepackage{hyperref}
\usepackage{mysec}
\usepackage{color}
\usepackage[normalem]{ulem}
\usepackage{cancel}

\makeatletter
\renewcommand{\@listI}%
{\leftmargin=\parindent
\partopsep=0pt
\parskip=-5pt
\topsep=2pt
\itemsep=-2pt
\labelwidth=\leftmargini }
\renewcommand{\@listii}{\setlength{\topsep}{2pt}
}  
 \makeatother

\makeatletter \@addtoreset {equation}{section}

\makeatother

\newtheorem{theorem}{Theorem}[section]
\newtheorem{lemma}[theorem]{Lemma}
\newtheorem{corollary}[theorem]{Corollary}
\newtheorem{proposition}[theorem]{Proposition}

\theoremstyle{remark}

\newtheorem{remark}{Remark}[section]
\newtheorem{assumption}{Assumption}[section]
\theoremstyle{definition}

\newcommand{\const}{{\mathop{\rm const}\nolimits}}

\newcommand{\myp}{\mbox{$\:\!$}}
\newcommand{\mypp}{\mbox{$\;\!$}}
\newcommand{\myn}{\mbox{$\;\!\!$}}
\newcommand{\mynn}{\mbox{$\:\!\!$}}

\newcommand{\PP}{P}
\newcommand{\QQ}{Q}

\newcommand{\EE}{E}
\newcommand{\Var}{\mathrm{Var}}
\newcommand{\Cov}{\mathrm{Cov}}
\renewcommand{\det}{\mathop{\rm det}\nolimits}
\newcommand{\supp}{\mathop{\rm supp}\nolimits}

\newcommand{\NN}{\mathbb{N}}
\newcommand{\ZZ}{\mathbb{Z}}
\newcommand{\RR}{\mathbb{R}}
\newcommand{\CC}{{\mathbb C}}

\newcommand{\calF}{\mathcal{F}}

\newcommand{\calH}{{\mathcal H}}

\newcommand{\CP}{\varPi}

\newcommand{\calX}{\mathcal{X}}

\newcommand{\rme}{{\mathrm e}}
\newcommand{\rmi}{{\mathrm i}}
\newcommand{\dif}{{\mathrm d}}

\newcommand{\topp}{{\mbox{\tiny$\top$}}}

\renewcommand{\footnotemark}{\fnsymbol{footnote}}

\begin{document}



\begin{center}
{\LARGE\bf Inverse problem of the limit shape for convex lattice
polygonal lines

}

\vspace{.8pc} {\large Leonid V.\ Bogachev\myp$^{\rm
a,\myn}$\footnote{Supported in part by a Leverhulme Research
Fellowship.} and Sakhavat M.\
Zarbaliev\myp$^{\rm b,}$\footnote{Supported in part by DFG Grant 436 RUS 113/722.}}\\[1pc]
\small $^{\rm a}$\myp Department of Statistics, School of
Mathematics,
University of Leeds, 
Leeds LS2 9JT, UK.\\
E-mail: {\tt L.V.Bogachev@leeds.ac.uk}\\[.4pc]
$^{\rm b}$\myp International Institute of Earthquake Prediction
Theory
and Mathematical Geophysics,\\
Russian Academy of Sciences, Profsoyuznaya Str. 84/32, Moscow
117997, Russia.\\ E-mail: {\tt szarbaliev@mail.ru}
\end{center}


%

\medskip\noindent
\begin{abstract}
\noindent
It is known that random convex polygonal lines on $\ZZ_+^2$
(with the endpoints fixed at $0=(0,0)$ and $n=(n_1,n_2)\to\infty$)
have a limit shape with respect to the uniform probability measure,
identified as the parabola arc
$\sqrt{c\myp(1-x_1)}+\sqrt{x_2}=\sqrt{c}$, where $n_2/n_1\to c$. The
present paper is concerned with the inverse problem of the limit
shape. We show that for any strictly convex, $C^3$-smooth arc
$\gamma\subset \RR_+^2$ starting at the origin, there is a
probability measure $\PP_n^\gamma$ on convex polygonal lines, under
which the curve $\gamma$ is their limit shape.




%
%
%
%

\end{abstract}

\medskip\noindent

\vspace{.5mm}

\section{Introduction}\label{sec1}
Consider a convex lattice polygonal line $\varGamma$ with vertices
on
$\ZZ_+^2:=\{(i,j)\in\ZZ^2: i,j\ge 0\}$, starting at the origin and
such that the slope of each of its edges is nonnegative and does not
exceed the angle of $90^\circ$.
Convexity means that the slope of consecutive edges is strictly
increasing. Let $\CP$ be the set of all such polygonal lines with
finitely many edges, and by $\CP_{n}$ the subset of polygonal lines
$\varGamma\in\CP$ with the right endpoint fixed at
$n=(n_1,n_2)\in\ZZ_+^2$.

The \textit{limit shape}, with respect to a probability measure
$\PP_n$ on $\CP_n$ as $n\to\infty$, is understood as a planar curve
$\gamma^*$ such that, for any $\varepsilon>0$,
\begin{equation}\label{eq:LLN}
\lim_{n\to\infty}\PP_n\{\varGamma\in\CP_n:
\,d(\tilde{\varGamma_n},\gamma^*)\le\varepsilon\}=1,
\end{equation}
where $\tilde\varGamma_n=S_n(\varGamma)$, subject to a suitable
scaling $S_n:\RR^2\to\RR^2$, and $d(\cdot,\cdot)$ is some metric on
the path space, e.g., induced by the Hausdorff distance between
compact sets,
\begin{equation}\label{eq:dH}
d_{\mathcal H}(A,B):=\max\left\{\max_{x\in A}\min_{y\in B}|x-y|,
\,\max_{y\in B}\min_{x\in A}|x-y|\right\}.
\end{equation}

Of course, the limit shape and its very existence may depend on the
probability law $\PP_n$. With respect to the uniform distribution on
$\CP_n$, the problem was solved independently by Vershik \cite{V1},
B\'ar\'any \cite{Barany} and Sinai \cite{Sinai}, who showed that if
$n_1,n_2\to\infty$ so that $n_2/n_1\to c\in(0,\infty)$, then under
the scaling $S_n:(x_1,x_2)\mapsto (x_1/n_1,x_2/n_2)$ limit
(\ref{eq:LLN}) holds with respect to the Hausdorff metric $d_\calH$
and with the limit shape $\gamma^*$ identified as a parabola arc
\begin{equation}\label{eq:gamma*}
\sqrt{c\myp(1-x_1)}+\sqrt{x_2}=\sqrt{c}, \qquad 0\le x_1,x_2\le1.
\end{equation}

Recently, Bogachev and Zarbaliev \cite{BZ4,BZ-DAN} proved that the
same limit shape (\ref{eq:gamma*}) appears for a large class of
measures $P_n$ of the form
\begin{equation}\label{eq:P-r}
\PP_n(\varGamma):= B_n^{-1} \prod_{e_i\in\varGamma} b_{k_i}
  \myp,\qquad B_n:={}\sum_{\varGamma\in\CP_n}\prod_{e_i\in\varGamma} b_{k_i}\qquad (\varGamma\in\CP_n),
\end{equation}
where the product is taken over all edges $e_i$ of
$\varGamma\in\CP_n$, $k_i$ is the number of lattice points on the
edge $e_i$ except its left endpoint, and
\begin{equation}\label{eq:b-k-r}
b_k:=\binom{r+k-1}{k}=\frac{r(r+1)\cdots(r+k-1)}{k!}\myp,\qquad
k=0,1,2,\dots.
\end{equation}
This result has provided first evidence in support of a conjecture
on the \textit{limit shape universality}, put forward independently
by Vershik \cite[p.\;20]{V1} and Prokhorov \cite{Prokhorov}. The
class of probability measures (\ref{eq:P-r}) with coefficients
(\ref{eq:b-k-r}) belongs to a general meta-type of decomposable
combinatorial structures known as \textit{multisets} (see
\cite[\S\myp2.2]{ABT}). Bogachev \cite{Bogachev} has extended the
universality result to a much wider class of multiplicative
probability measures (\ref{eq:P-r}) including the analogues of two
other well-known meta-types of decomposable structures ---
\textit{selections} and \textit{assemblies} (cf.\
\cite[\S\myp2.2]{ABT}); for example, this class includes the uniform
distribution on the subset of ``simple'' polygonal lines (i.e.,
those with no lattice points apart from vertices).

However, universality of the limit shape $\gamma^*$ given by
(\ref{eq:gamma*}) has its boundaries; indeed, in the present paper
we consider the inverse problem and show that \textit{any
$C^3$-smooth, strictly convex arc $\gamma\in\RR_+^2$} (started at
the origin) may appear as the limit shape with respect to a suitable
probability measure $\PP_n^\gamma$ on $\CP_n$, as $n\to\infty$. For
early drafts of this result (treated in terms of approximation of
convex curves by random polygonal lines), see \cite{BZ2,BZ3}.

Like in \cite{Bogachev,BZ4}, our construction employs an elegant
probabilistic approach based on randomization and conditioning (see
\cite{ABT}), first used in the polygonal context by Sinai
\cite{Sinai}. The idea is to represent a given measure $\PP_n$ on
$\CP_{n}$ as the conditional distribution,
$\PP_n(\varGamma)=\QQ(\varGamma\,|\,\CP_n)$, induced by a suitable
``global'' probability measure $\QQ$ defined on the space
$\CP=\cup_n\CP_n$ of convex lattice polygonal lines with a free
right endpoint. In turn, the measure $\QQ=\QQ_{z}$ depending on a
two-dimensional parameter $z=(z_1,z_2)$ is constructed as the
distribution of a suitable integer-valued random field
$\nu=\nu(\cdot)$ with mutually independent components, defined on
the subset $\mathcal{X}\subset\ZZ_+^2$ consisting of points
$x=(x_1,x_2)\in\ZZ_+^2$ with co-prime coordinates. Note that a
polygonal line $\varGamma\in\CP$ is easily retrieved from a
configuration $\{\nu(x)\}_{x\in\calX}$ using the collection of the
corresponding edges $\{x\nu(x)\}_{x\in\calX}$ and the convexity
property.

It turns out, however, that in order to fit a given curve $\gamma$
the parameter $z=(z_1,z_2)$ needs to allow for a dependence on $x\in
X$. We derive suitable parameter functions $z_1(x)$ and $z_2(x)$,
assuming that they depend on $x$ through the ratio $x_2/x_1$ only,
which is particularly convenient in conjunction with the
parameterization of the curve $\gamma$ using its tangent slope. As
one would anticipate, if $\gamma=\gamma^*$ then the functions
$z_1(x)$, $z_2(x)$ are reduced to constants and our method recovers
the uniform distribution on $\CP_n$.

To summarize, our main result is as follows.
\begin{theorem}\label{th:main1}
Let $\gamma\subset\RR_+^2$  be a strictly convex, $C^3$-smooth
arc\textup{,} with the endpoints\/ $0$ and $(1,c_\gamma)$ and with
the curvature bounded from below by a positive constant. Suppose
that\/ $n_2/n_1 \to c_\gamma$\textup{,} and set
$\tilde\varGamma_n:=n_1^{-1}\varGamma$. Then there is a probability
measure $\PP_n^\gamma$ on $\CP_n$ such that\textup{,} for any\/
$\varepsilon>0$\textup{,}
\begin{equation}\label{LLNappr}
\lim_{n\to\infty}\PP_n^\gamma\{\varGamma\in\CP_n:\,d_{\mathcal
H}(\tilde\varGamma_n,\gamma)\le\varepsilon\}=1.
\end{equation}
\end{theorem}
\begin{remark}
It will be more convenient to use another metric on the space of
convex paths (denoted by $d_{\mathcal{L}}$), based on the tangential
parameterization of paths and a sup-distance between the
corresponding arc lengths. However, the metrics $d_{\mathcal{L}}$
and $d_{\mathcal{H}}$ are equivalent.
\end{remark}

\begin{remark}
It is interesting to try and relax the $C^3$-smoothness condition on
$\gamma$ (e.g., by permitting ``change-points'' or corners), as well
as to allow for degeneration of the curvature (e.g., through
possible flat segments). We will address these issues elsewhere.
\end{remark}

The rest of the paper is organized as follows. In
Section~\ref{sec2}, we introduce the space of convex paths on the
plane and endow it with a suitable metric. In Section \ref{sec3},
the measures $\QQ_z^\gamma$ and $\PP_n^\gamma$ are constructed for a
given convex curve $\gamma$. In Section~\ref{sec4}, the parameter
vector-function $z(x)$ is chosen to guarantee the convergence of
``expected'' scaled polygonal lines
$\tilde\varGamma_n=n_1^{-1}\varGamma$ to the target curve $\gamma$
(Theorem~\ref{th:1}). Refined first-order moment asymptotics are
obtained in Section~\ref{sec5}, while higher-order moment sums are
analyzed in Section~\ref{sec6}. Section~\ref{sec7} is devoted to the
proof of a local central limit theorem (Theorem~\ref{llt}). Finally,
the limit shape result, with respect to both $\QQ_z^\gamma$ and
$\PP_n^\gamma$, is proved in Section~\ref{sec8} (Theorems \ref{zbc1}
and~\ref{zbc2}, respectively).

\subsubsection*{Some
notations.}
For a row-vector $x=(x_1,x_2)\in \RR^2$, its Euclidean norm (length)
is denoted by $|x|:=(x_1^2+x_2^2)^{1/2}$, and $\langle
x,y\rangle:=x\mypp y^{\myn\topp\!}=x_1y_1+x_2\myp y_2$ is the
corresponding inner product of vectors $x,y\in \RR^2$. We denote
$\ZZ_+:=\{k\in\ZZ: k\ge0\}$, \,$\ZZ_+^2:=\ZZ_+\mynn\times\ZZ_+$\myp,
and similarly $\RR_+:=\{x\in\RR: x\ge0\}$,
$\RR_+^2:=\RR_+\mynn\times\RR_+$\myp. For $z=(z_1,z_2)\in\RR_+^2$
and $x=(x_1,x_2)\in\ZZ_+^2$, we use the multi-index notation
$z^{x}:=z_1^{x_1}z_2^{x_2}$. The notation $a_n\asymp b_n$ as
$n\to\infty$ means that
$$
0<\liminf_{n\to\infty}\frac{x_n}{y_n}\le\limsup_{n\to\infty}
\frac{x_n}{y_n}<\infty.
$$
We also use the standard notation $x_n\sim y_n$ for $x_n/y_n\to1$ as
$n_1,n_2\to\infty$.

\section{Preliminaries: convex planar paths}\label{sec2}
Let $g$ be a bounded function defined on some interval $[0,a]$, such
that $g(0)=0$, and suppose that $g$ is non-decreasing and convex on
$[0,a]$. Convexity means that the function's epigraph $\{(u,v): 0\le
u\le a, \,g(u)\le v\}$ is a convex set on the plane. Furthermore,
assume that $g$ is continuous on $[0,a]$ and piecewise
differentiable, with the derivative $g'$ continuous everywhere
except at a finite set of points (we allow $g'(a)$ to be infinite,
$g'(a)\le+\infty$).
It follows that the function $t=g'(u)$ is nonnegative and
non-decreasing in its domain, and in particular $0\le t_0\le
g'(u)\le t_1\le\infty$ \,($0\le u\le a$), where
\begin{equation}\label{eq:t_01}
  t_0:=\inf_{0\le u\le a} g'(u),\qquad t_1:=\sup_{0\le u\le a} g'(u).
\end{equation}

Denote by $\gamma_g\equiv\gamma$
the graph of a function $g$ with the above properties, and let
$\mathfrak{G}$ be the set of all such curves. For the spaces
$\CP_n$, $\CP$ of convex polygonal lines introduced ablve, we have
$\CP_n\subset\CP\subset\mathfrak G$. If a polygonal line
$\varGamma\in\CP_n$ is taken as a ``curve'' $\gamma$, then the
corresponding function $g=g_\varGamma$ is a piecewise linear
function defined on $[0,n_1]$, such that $g_\varGamma(n_1)=n_2$.

Let us now equip the space $\mathfrak{G}$ with a suitable metric. If
a function $g=g_\gamma$ determines a convex curve
$\gamma\in\mathfrak{G}$, we set
\begin{equation}\label{u_g}
u_\gamma(t):=\sup\{u: g'_\gamma(u)\le t\},\qquad 0\le t\le\infty,
\end{equation}
with the convention that $\sup\emptyset=0$. That is, $u_\gamma(t)$
is a
generalized inverse of the derivative $t=g'_\gamma(u)$
(cf.\ \cite[\S1.5]{BGT}).
It follows that the function $u_\gamma(\cdot)$ is non-decreasing and
right-continuous on $[0,\infty]$, with values in $[0,a]$. Moreover,
if $t_0$, $t_1$ are the extreme values of the derivative $g'_\gamma$
(see (\ref{eq:t_01})) then $u_\gamma(t)\equiv 0$ for all $t<t_0$ and
$u_\gamma(t)=a$ for all $t\ge t_1$.

Denote by $\ell_{\gamma}(t)$ the length of the part of $\gamma$
where the tangent slope does not exceed $t$,
\begin{equation}\label{ell}
\ell_\gamma(t)=\int_0^{u_\gamma(t)}\sqrt{1+g'_\gamma(u)^2}\;\dif
u,\qquad 0\le t\le\infty.
\end{equation}
Our assumptions imply that every curve $\gamma\in\mathfrak{G}$ is
rectifiable, since
$$
 \ell_\gamma(\infty)=\int_0^{u_\gamma(\infty)}\sqrt{1+g'_\gamma(u)^2}\,\dif u\le
 \int_0^a(1+g'_\gamma(u))\,\dif u=a+g_\gamma(a)<\infty.
$$
Finally, we define the function $d_{\mathcal{L}}:\mathfrak
G\times\mathfrak G\to\RR_+\cup\{+\infty\}$ by setting
\begin{equation}\label{eq:rho}
d_{\mathcal{L}}(\gamma_1,\gamma_2):=\sup_{0\le t\le\infty}
|\ell_{\gamma_1}(t)-\ell_{\gamma_2}(t)|,\qquad
\gamma_1,\gamma_2\in\mathfrak{G}.
\end{equation}

\begin{proposition}\label{pr:d1}
The function $d_{\mathcal{L}}(\cdot,\cdot)$ satisfies all properties
of a distance.
\end{proposition}

\begin{proof}
Clearly,
$d_{\mathcal{L}}(\gamma_1,\gamma_2)=d_{\mathcal{L}}(\gamma_2,\gamma_1)$
and $d_{\mathcal{L}}(\gamma,\gamma)=0$. The triangle axiom is also
obvious.
So it remains to verify that if
$d_{\mathcal{L}}(\gamma_1,\gamma_2)=0$ then $\gamma_1=\gamma_2$.
To this end, approximating $\gamma_1,\gamma_2\in\mathfrak G$ by
$C^2$-smooth strictly convex curves $\gamma_1^k$, $\gamma_2^k$,
respectively, we reduce the problem to checking that if
$\gamma_1^k$, $\gamma_2^k$ are close to each other in the sense of
$d_{\mathcal{L}}$, then they are also close in the Hausdorff metric
$d_{\mathcal{H}}$; that is, if
$d_{\mathcal{L}}(\gamma_1^k,\gamma_2^k)\to0$ then
$d_{\mathcal{H}}(\gamma_1^k,\gamma_2^k)\to0$ \,($k\to\infty$).

Next, for a strictly convex, increasing function $g_\gamma\in
C^2[0,a]$, the function $u_\gamma(t)$ defined in (\ref{u_g}) is
given explicitly by
\begin{equation}\label{eq:u(t)}
u_\gamma(t)=\left\{\begin{array}{ll}
\ 0,&0\le t\le t_0,\\
(g'_\gamma)^{-1}(t),\ & t_0\le t\le t_1,\\
\ a,&t_1\le t\le\infty,
\end{array}
\right.
\end{equation}
where $(g'_\gamma)^{-1}(t)$ is the (ordinary) inverse of the
derivative $g'_\gamma(u)$. In particular, the equations
$u=u_\gamma(t)$, \,$v=g_\gamma(u_\gamma(t))$ determine a
parameterization of the curve $\gamma$ via the derivative
$t=g'_\gamma(u)$. Differentiating formula (\ref{ell}) with respect
to $t$, we find
\begin{align}\label{du}
\frac{\dif u_\gamma}{\dif t}=\frac{1}{\sqrt{1+t^2}}\,\frac{\dif
\ell_\gamma}{\dif t}\myp,\qquad u_\gamma(0)=0,
\end{align}
and hence
\begin{align}\label{dv}
\frac{\dif v_\gamma}{\dif t}=\frac{\dif g_\gamma}{\dif
u}\cdot\frac{\dif u_\gamma}{\dif t}=
\frac{t}{\sqrt{1+t^2}}\,\frac{\dif \ell_\gamma}{\dif t}\myp,\qquad
v_\gamma(0)=0.
\end{align}
Integrating equations (\ref{du}), (\ref{dv}) by parts, we obtain
\begin{align} \label{u}
u_\gamma(t)=\frac{\ell_\gamma(t)}{\sqrt{1+t^2}}+
\int_0^t\frac{s\myp\ell_\gamma(s)}{(1+s^2)^{3/2}}\,\dif s,\ \ \quad
v_\gamma(t)=\frac{t\myp\ell_\gamma(t)}{\sqrt{1+t^2}}-
\int_0^t\frac{\ell_\gamma(s)}{(1+s^2)^{3/2}}\,\dif s.
\end{align}
Note that these equations are linear in $\ell_\gamma$. Hence,
setting for $\gamma_1^k$, $\gamma_2^k$
\begin{gather*}
  \Delta u_k(t):=u_{\gamma_1^k}(t)-u_{\gamma_2^k}(t),\quad
  \Delta v_k(t):=v_{\gamma_1^k}(t)-v_{\gamma_2^k}(t),\\
  \Delta \ell_k(t):=\ell_{\gamma_1^k}(t)-\ell_{\gamma_2^k}(t),
\end{gather*}
from (\ref{u}) we get
\begin{equation*}
\Delta u_k(t)= \frac{\Delta\ell_k(t)}{\sqrt{1+t^2}}+ \int_0^t
\frac{s\myp\Delta\ell_k(s)}{(1+s^2)^{3/2}}\,\dif s,\quad \Delta
v_k(t)=\frac{t\myp\Delta\ell_k(t)}{\sqrt{1+t^2}}-
\int_0^t\frac{\Delta\ell_k(s)}{(1+s^2)^{3/2}}\,\dif s.
\end{equation*}
This implies that if $\Delta\ell_k(t)\to 0$, uniformly in
$t\in[0,\infty]$, then $\Delta u_k(t)\to 0$, $\Delta v_k(t)\to 0$,
also uniformly on $[0,\infty]$ \,($k\to\infty$). This completes the
proof.
\end{proof}

From the proof of Proposition \ref{pr:d1}, one can see that the following
result holds.
\begin{corollary}\label{cor:d-d1}
The metrics\/ $d_{\mathcal{L}}$ and\/ $d_{\mathcal{H}}$ are
equivalent\textup{;} in particular\textup{,}
$d_{\mathcal{L}}(\gamma_k,\gamma)\to 0$ if and only if\/
$d_{\mathcal{H}}(\gamma_k,\gamma)\to 0$.
\end{corollary}

Consider a fixed convex curve $\gamma\in\mathfrak G$, represented as
the graph of an increasing convex function $g_\gamma$, which for
definiteness is assumed to be defined on the interval $[0,1]$. We
will be working under the following
\begin{assumption}\label{as2}
The function $g_\gamma$ is strictly increasing and strictly convex
on $[0,1]$, and $g_\gamma\in C^2[0,1]$. In particular,
$g'_\gamma(u)\ge 0$, $g''_\gamma(u)\ge 0$ for all $u\in[0,1]$.
Moreover, the curvature $\varkappa_\gamma$ of the curve $\gamma$,
given by the formula
\begin{equation}\label{eq:curvature}
\varkappa_\gamma(u)=\frac{g''_\gamma(u)}{(1+g'_\gamma(u)^2)^{3/2}}\myp,\qquad
0\le u\le 1,
\end{equation}
is uniformly bounded from below,
\begin{align}\label{kr}
\inf_{u\in[0,1]}\varkappa_\gamma(u)\ge K_0 >0.
\end{align}
\end{assumption}

As was mentioned in the proof of Proposition \ref{pr:d1}, the graph
$\gamma$ of the function $g_\gamma$ can be parameterized by the
derivative $t=g_\gamma'(u)$ via the equations $u=u_\gamma(t)$,
$v=g_\gamma(u_\gamma(t))$, where $u_\gamma(t)$ is given by
(\ref{eq:u(t)}). Expression (\ref{eq:curvature}) for the curvature
is then reduced to
\begin{equation}\label{eq:curvature1}
\varkappa_\gamma(t)=\frac{g''_\gamma(u_\gamma(t))}{(1+t^2)^{3/2}}\myp,\qquad
t_0\le t\le t_1,
\end{equation}
where $t_0=\inf_u g'_\gamma(u)$, \,$t_1=\sup_u g'_\gamma(u)$ (see
(\ref{eq:t_01})).

\section{Construction of the measures $\QQ_z^{\gamma}$ and
$\PP_{n}^\gamma$}\label{sec3}


Consider the set $\mathcal{X}\subset\ZZ_+^2$ of all pairs of
co-prime nonnegative integers,
\begin{equation}\label{eq:X}
\mathcal{X}:=\{x=(x_1,x_2)\in\ZZ_+^2:\ \gcd(x_1,x_2)=1\},
\end{equation}
where ``$\gcd$'' stands for ``greatest common divisor''.
Denote by $\tau(x):=x_2/x_1\in[0,+\infty]$ the slope of the vector
$x=(x_1,x_2)\in\mathcal{X}$. Let $\varPhi:= (\ZZ_+)^{\calX}$ be the
space of functions on $\calX$ with nonnegative integer values, and
consider the subspace of functions with \textit{finite support}\/,
$\varPhi_0 :=\{\nu\in\varPhi :\,\#(\supp\myp\nu)<\infty\}$, where $
\supp\myp\nu:=\{x\in \calX:\,\nu(x)>0\}$. It is easy to see that the
space $\varPhi_0$ is in one-to-one correspondence with the space
$\CP=\bigcup_{n\in\ZZ_+^2}\CP_n$ of all (finite) convex lattice
polygonal lines, whereby each $x\in \calX$ determines the
\emph{direction} of a potential edge, only utilized if
$x\in\supp\myp\nu$, in which case the value $\nu(x)>0$ specifies the
\emph{scaling factor}, altogether yielding a vector edge
$x\myp\nu(x)$; finally, assembling all such edges into a polygonal
line is uniquely determined by the fixation of the starting point
(at the origin) and the convexity property.
%
Note that $\nu(x)\equiv0$ formally corresponds to the ``trivial''
polygonal line with coinciding endpoints. In what follows, we
identify the spaces $\CP$ and $\varPhi_0(\mathcal{X})$.

Let us now introduce on $\varPhi_0(\mathcal{X})$ a probability
measure $\QQ_{z}^{\gamma}$ by setting
\begin{align}\label{a1}
\QQ_{z}^{\gamma}(\varGamma):=&{}\prod_{x\in \mathcal{X}} z^{x
\nu(x)} (1-z^{x}),\qquad
\CP\ni\varGamma\leftrightarrow \nu\in\varPhi_0(\mathcal{X}),
\end{align}
where $z\equiv z(x)=(z_1(x),z_2(x))$ is a parameter
(vector-)function such that $0\le z_j(x)<1$ \,($x\in \mathcal{X}$).
Its explicit form, determined by a given curve
$\gamma\in\mathfrak{G}$, will be specified later on. So far, we only
assume that
\begin {align}\label{eq:norm}
Z:=\prod_{x\in \mathcal{X}} (1-z^{x})>0,
\end{align}
which guarantees that
(\ref{a1}) is well
defined. Definition (\ref{a1}) implies that the random variables
$\{\nu(x)\}_{x\in \mathcal{X}}$ are mutually independent and have
geometric distribution with parameter $z^{x}$,
\begin{align}\label{gz}
\QQ_z^{\gamma}\{\nu(x)=k\}=z^{kx}(1-z^{x}),\quad k\in\ZZ_+;
\end{align}
in particular, the corresponding expected value and variance are
given by \cite[\S\myp{}XI.2, p.~269]{Feller}
\begin{equation}\label{eq:nu}
\EE_{z}^{\gamma}[\nu(x)]=\frac{z^x}{1-z^x}=\sum_{k=1}^\infty
z^{kx},\qquad \Var[\nu(x)]=\frac{z^x}{(1-z^x)^2}=\sum_{k=1}^\infty k
z^{kx}.
\end{equation}

 Note that
$\QQ_{z}^{\gamma}$ can be extended in a standard way to a measure on
the space $\varPhi(\mathcal{X})$ of all nonnegative integer-valued
functions on $\mathcal{X}$.
However, $\QQ_z^{\gamma}$ is in
fact concentrated on the subset $\varPhi_0(\mathcal{X})\subset
\varPhi(\mathcal{X})$ consisting of all finite configurations
$\nu(\cdot)$.

\begin{lemma}\label{lm:Q=1}
Condition\/ \textup{(\ref{eq:norm})} is necessary and sufficient in
order that\/ $\QQ_z^{\gamma}\{\nu\in\varPhi_0(\mathcal{X})\}=1$.
\end{lemma}

\begin{proof}
According to (\ref{gz}),
$$
\sum_{x\in \mathcal{X}} \QQ_z^{\gamma}\{\nu(x)>0\}=\sum_{x\in
\mathcal{X}}z^{x}<\infty
$$
whenever
the infinite product in (\ref{eq:norm}) is convergent. By
Borel--Cantelli's lemma, this implies  that only finitely many
events $\{\nu(x)>0\}$ may occur ($\QQ_z^{\gamma}$-a.s),
and the lemma is proved.
\end{proof}

As a result, with $\QQ_z^\gamma$-probability $1$ a realization of
the random field $\nu(\cdot)$ determines a (random) convex polygonal
line $\varGamma\in\CP$. Denote by $\xi_\varGamma=(\xi_1,\xi_2)$ the
right endpoint of $\varGamma$,
%
so that $\CP_n=\{\varGamma\in\CP: \,\xi_\varGamma=n\}$.
Accordingly,
$\QQ_z^{\gamma}$ induces the conditional distribution
$\PP_{n}^{\gamma}$ on $\CP_n$,
\begin{align}\label{u_1}
\PP_{n}^{\gamma}(\varGamma):=\QQ_{z}^{\gamma}\{\varGamma\,|\,\xi_\varGamma=n\}=
\frac{\QQ_{z}^{\gamma}(\varGamma)}{\QQ_{z}^{\gamma}\{\xi_\varGamma=n\}}\mypp,
\qquad \varGamma\in\CP_{n}.
\end{align}

\section{The choice of the parameter function $z(x)$}\label{sec4}

In the above construction, the measure $\PP_z^{\gamma}$ depends on
the vector parameter $\{z(x)\}_{x\in \mathcal{X}}$. So far, this
function was only assumed to guarantee the convergence of the
infinite product in (\ref{eq:norm}). Let us now adjust it to a given
curve $\gamma\in\mathfrak{G}$.

Let $\varGamma(t)$ denote the part of the polygonal line
$\varGamma\in\CP$ where the slope of edges does not exceed
$t\in[0,\infty]$. Set $\mathcal{X}(t):=\{x\in \mathcal{X}:
\tau(x)\le t\}$. Recalling the association $\varGamma\leftrightarrow
\nu$ described in Section \ref{sec3}, the polygonal line
$\varGamma(t)$ is determined by the truncated configuration
$\nu(x)\mypp{\mathbf 1}_{\mathcal{X}(t)}(x)$. Denote by
$\xi(t)=(\xi_1(t),\xi_2(t))$ the right endpoint of $\varGamma(t)$,
\begin{align}\label{eq:u,v_Gamma}
\xi_j(t)=\sum_{x\in \mathcal{X}(t)}x_j\myp\nu(x)\qquad (j=1,2),
\end{align}
and by $\ell_\varGamma(t)$ its length,
\begin{align}\label{ell_Gamma}
\ell_{\varGamma}(t)=\sum_{x\in \mathcal{X}(t)}|x|\mypp\nu(x).
\end{align}

Let us impose the following calibration condition,
\begin{align}\label{lim}
\lim_{n\to\infty}n_1^{-1}\EE_{z}^{\gamma}[\ell_{\varGamma}(t)]=
\ell_\gamma(t),\qquad 0\le t\le \infty,
\end{align}
where $\EE_z^{\gamma}$ stands for the expectation with respect to
the measure $\QQ_z^{\gamma}$ and $\ell_\gamma(t)$ is the
corresponding length function associated with a given curve
$\gamma$. More specifically, denote $\rho_n:=c_\gamma/c_n\to 1$ and
set $\tilde x:=(x_1,\rho_n\myp x_2)$ for $x=(x_1,x_2)$. We will seek
$z_1(x)$, $z_2(x)$ in the form
\begin{equation}\label{z}
z_j(x)=\exp\left\{-\alpha_n\myp\delta_j(\tau(\tilde
x))\right\}\qquad (j=1,2),
\end{equation}
where
\begin{equation}\label{eq:alpha}
\alpha_n:=(\rho_n\myp n_1)^{-1/3}\to 0,\qquad \tau(\tilde x)=\tilde
x_2/\tilde x_1=\rho_n\myp x_2/x_1,
\end{equation}
and $\delta(t)=(\delta_1(t),\delta_2(t))$ is a function on
$[0,\infty]$ such that
\begin{equation}\label{ogd}
\inf_{0\le t\le\infty}\delta_j(t)\ge\delta_{*\myn}>0\qquad (j=1,2).
\end{equation}

\begin{remark}\label{as1}
Note that the right endpoint of the scaled polygonal line
$\tilde\varGamma_n:=n_1^{-1}\varGamma$ ($\varGamma\in\CP$) has the
coordinates $(1,c_n)$, where $c_n:=n_2/n_1$, whereas the right
endpoint of the arc $\gamma$ lies at the point $(1,c_\gamma)$, where
$c_\gamma:=g_\gamma(1)$ \,($0<c_\gamma<\infty$). Hence, in order for
relation (\ref{LLNappr}) to be true, it is natural to pass to the
limit $n\to\infty$ in such a way that $c_n\to c_\gamma$. In what
follows, we will always be assuming that this condition is
fulfilled.
\end{remark}
According to (\ref{eq:nu}) and (\ref{ell_Gamma}), and using notation
(\ref{z}), we have
\begin{equation}\label{per+}
\EE_{z}^{\gamma}[\ell_{\varGamma}(t)]=\sum_{x\in \mathcal{X}(t)}
|x|\sum_{k=1}^\infty z^{kx}
=\sum_{k=1}^\infty\sum_{x\in \mathcal{X}(t)} |x|\, \rme^{-\alpha_n
k\langle \tilde x,\myp\delta(\tau(\tilde x))\rangle}.
\end{equation}

 To deal with sums over
the sets $\calX(t)\subset \calX$, the following lemma will be
instrumental. Recall that the \textit{M\"obius function} $\mu(m)$
($m\in\NN$) is defined as follows: $\mu(1):=1$, \,$\mu(m):=(-1)^{d}$
if $m$ is a product of $d$ different prime numbers, and $\mu(m):=0$
otherwise (see \cite[\S16.3, p.\:234]{HW}); in particular,
$|\mu(m)|\le 1$ for all $m\in\NN$.
\begin{lemma}\label{lm:Mobius}
Let $f:\RR_+^2\to\RR$ be a function such that $f(0,0)=0$ and
\begin{equation}\label{eq:sum_sum<}
\sum_{k=1}^\infty\sum_{x\in\ZZ_+^2} |f(hkx)|<\infty,\qquad h>0.
\end{equation}
For $h>0$, consider the functions
\begin{gather}\label{f+lemma}
F(h):=
\sum_{m=1}^\infty\sum_{x\in
\calX} f(hmx),\qquad
F^\sharp(h):=\sum_{x\in \calX} f(hx).
\end{gather}
Then the following identities hold for all $h>0$
\begin{align}
F(h)&= \sum_{x\in\ZZ^2_+}f(hx),\qquad \label{eq:Mobius}
F^\sharp(h)=\sum_{m=1}^\infty\mu(m)F(hm).
\end{align}
\end{lemma}
\begin{proof}
Recalling definition (\ref{eq:X}) of the set $\calX$, observe that
$\ZZ^2_+=\bigsqcup_{\myp m=0}^{\myp\infty} m\calX$; hence,
definition (\ref{f+lemma}) of $F(\cdot)$ is reduced to
(\ref{eq:Mobius}). Representation (\ref{eq:Mobius}) for
$F^\sharp(\cdot)$ follows from the M\"obius inversion formula (see
\cite[Theorem~270, p.\:237]{HW}), provided that $\sum_{k,\myp{}m}
|F^\sharp(hkm)|<\infty$. To verify the last condition, using
(\ref{f+lemma}) we obtain
\begin{align*}
\sum_{k,\myp{}m=1}^\infty |F^\sharp(kmh)|&\le \sum_{k=1}^\infty
\left(\sum_{m=1}^\infty\sum_{x\in \calX} |f(hkmx)|\right)
=\sum_{k=1}^\infty \sum_{x\in \ZZ^2_+} |f(hkx)|<\infty,
\end{align*}
according to (\ref{eq:sum_sum<}).
\end{proof}

\begin{theorem}\label{th:1}
Suppose that the functions\/ $\delta_1(t)$\textup{,} $\delta_2(t)$
satisfy condition\/ \textup{(\ref{ogd})}. Then\textup{,} in order
that equation\/ \textup{(\ref{lim})} be fulfilled for all\/
$t\in[0,\infty]$\textup{,} it is necessary and sufficient that\/
\begin{align}
\label{eq:delta+infty}
\delta_j(t)&\equiv+\infty\quad (j=1,2),&&\hspace{-4pc}t<t_0,\ \ t>t_1,\\
\label{eq:delta-eqn} \delta_1(t)&+t\delta_2(t)=\kappa\myp
g''_\gamma(u_\gamma(t))^{1/3}, &&\hspace{-4pc}t_0<t<t_1,
\end{align}
where\/ $\kappa:=(2\zeta(3)/\zeta(2))^{1/3}$\textup{,}
$\zeta(s):=\sum_{k=1}^\infty 1/k^{s}$ is the Riemann zeta
function\textup{,} and the function\/ $u_\gamma(t)$ is given by\/
\textup{(\ref{eq:u(t)})}.
\end{theorem}

\begin{proof}
Let us set
\begin{equation}\label{f+}
f(x):=|x|\,\rme^{-\alpha_n \langle
  \tilde x,\myp\delta(\tau(\tilde x))\rangle}\mypp{\bf1}_{\calX(t)}(x),\qquad x\in\RR^2_+\myp,
\end{equation}
suppressing for simplicity the dependence on $t$ and $n$. Following
notations (\ref{f+lemma}) of Lemma \ref{lm:Mobius}, equation
(\ref{per+}) is rewritten in the form
\begin{equation}\label{E_1F}
\EE_{z}^{\gamma}[\ell_{\varGamma}(t)] =\sum_{k=1}^\infty k^{-1}
F^\sharp(k),
\end{equation}
whereas from (\ref{eq:Mobius}) we have
\begin{align}\label{F}
F(h)&=\sum_{x_1=0}^\infty\,\sum_{0\le x_2\le tx_1}
h|x|\,\rme^{-\alpha_n h \langle \tilde x,\myp\delta(\tau(\tilde
x))\rangle}\\
\notag &\le h\sum_{x_1\myn,\myp x_2=0}^\infty (x_1+x_2) \mypp
\rme^{-\alpha_n h\delta_{*}(x_1+x_2)/2}= h\sum_{y=0}^\infty y^2\mypp
\rme^{-\alpha_n h\delta_{*\myn} y/2}\\
&=h\,\frac{\rme^{-\alpha_n h\delta_{*\myn}/2}+\rme^{-\alpha_n
h\delta_{*}}} {(1-\rme^{-\alpha_n
h\delta_{*\myn}/2})^3}=O(1)\,\alpha_n^{-3} h^{-2}.
\label{eq:hkm}
\end{align}
In particular, this gives $F(hk)=O(k^{-2})$, uniformly in $k\in\NN$,
and it follows that condition (\ref{eq:sum_sum<}) of Lemma
\ref{lm:Mobius} is satisfied. Hence, using
(\ref{eq:Mobius}) and (\ref{F}) and recalling that
$n_1^{-1}=\rho_n\myp\alpha_n^3$, from (\ref{E_1F}) with $h=k$ we
obtain
\begin{align}
\notag n_1^{-1}\EE_{z}^{\gamma}[\ell_{\varGamma}(t)]&=
\rho_n\alpha_n^3\sum_{k,\myp m=1}^\infty m\mu(m)F(km)\\
\label{per-mu} &=\rho_n\sum_{k,\myp m=1}^\infty m\mu(m)
\sum_{x_1=1}^\infty\,\sum_{0\le x_2\le tx_1}
\alpha_n^3|x|\,\rme^{-km\alpha_n\langle\tilde
x,\myp\delta(\tau(\tilde{x}))\rangle}.
\end{align}

Taking into account estimate (\ref{eq:hkm}),
we see that the general term in the double sum over $k,m$ in
(\ref{per-mu}) admits a uniform bound of the form
$O(1)\,k^{-3}m^{-2}$, which is a term of a convergent series.
Therefore, we can apply Lebesgue's dominated convergence theorem to
pass to the limit in (\ref{per-mu}) termwise, as $n\to\infty$ (i.e.,
$\alpha_n\to0$). In order to find this limit, note that the internal
double series over $x_1$, $x_2$ in (\ref{per-mu}) is a Riemann sum
for the integral
\begin{equation}\label{int}
\iint_{0\le x_2\le tx_1}\!\sqrt{x_1^2+x_2^2}\,\, \rme^{-km\myp
\bigl(x_1\delta_1(x_2/x_1)+x_2\myp\delta_2\myp(x_2/x_1)\bigr)}\,\dif
x_1 \dif x_2.
\end{equation}
Moreover, this sum does converge to integral (\ref{int}) as
$\alpha_n\to0$, since the integrand function in (\ref{int}) is
directly Riemann integrable, as follows from an estimation similar
to (\ref{eq:hkm}).

By the change of variables $y_1=u$, $y_2=us$ integral (\ref{int})
is reduced to
\begin{equation*}
\int_0^t \sqrt{1+s^2}\left(\int_0^\infty
u^2\,\rme^{-kmu\myp(\delta_1(s)+s\delta_2(s))}\,\dif u \right)
\dif{s}=\frac{2}{(km)^3}\int_0^t\frac{\sqrt{1+s^2}}
{\bigl(\delta_1(s)+s\delta_2(s)\bigr)^3}\,\dif s.
\end{equation*}
Substituting this into (\ref{per-mu}) we get
\begin{align}
\notag
  \lim_{n\to\infty} n_1^{-1}\EE_{z}^{\gamma}[\ell_{\varGamma}(t)]&=
2\sum_{k=1}^\infty\frac{1}{k^3}\sum_{m=1}^\infty \frac{\mu(m)}{m^2}
\int_0^t\frac{\sqrt{1+s^2}}
{\bigl(\delta_1(s)+s\delta_2(s)\bigr)^3}\,\dif s\\
&=\frac{2\zeta(3)}{\zeta(2)} \int_0^t\frac{\sqrt{1+s^2}}
{\bigl(\delta_1(s)+s\delta_2(s)\bigr)^3}\,\dif s,
\label{eq:sqrt(1+s2)}
\end{align}
where we used the identity $\sum_{m=1}^\infty
m^{-2}\mu(m)=\zeta(2)^{-1}$, which readily follows by the M\"obius
inversion formula (\ref{eq:Mobius}) applied to $F^\sharp(h)=h^{-2}$,
\,$F(h)=\sum_{m=1}^\infty (hm)^{-2} =h^{-2}\myp\zeta(2)$ \,(cf.\
\cite[\S17.5, Theorem~287, p.\;250]{HW}). Recalling the notation
$\kappa$ introduced in Theorem \ref{th:1} and using condition
(\ref{lim}), from (\ref{eq:sqrt(1+s2)}) we obtain
\begin{align}\label{fu1}
\kappa^3\!\int_0^t\frac{\sqrt{1+s^2}}
{\bigl(\delta_1(s)+s\delta_2(s)\bigr)^3} \,\dif
s=\ell_{\gamma}(t)\myp,\qquad 0\le t\le \infty.
\end{align}
According to definitions (\ref{u_g}) and (\ref{ell}), we have
$\ell_\gamma(t)\equiv0$ for $t\in[0,t_0)$ and
$\ell_\gamma(t)\equiv\ell_\gamma(\infty)$ for $t\in(t_1,\infty]$,
while for $t\in(t_0,t_1)$ the derivative $\ell'_\gamma(t)$ is
determined by formula (\ref{du}). Hence, differentiating identity
(\ref{fu1}) with respect to $t$, we obtain  (\ref{eq:delta+infty})
and (\ref{eq:delta-eqn}).
\end{proof}

Let us now check that equation (\ref{eq:delta-eqn}) has a suitable
solution.

\begin{proposition}\label{pr:delta}
For\/ $t\in [t_0,t_1]$ let us set
\begin{equation}\label{eq:delta12}
\delta_1(t):=\kappa\mypp\varkappa_\gamma(t)^{1/3}\,
\frac{c_\gamma\myp\sqrt{1+t^2}}{c_\gamma+t}
\myp,\qquad \delta_2(t):=
\frac{\delta_1(t)}{c_\gamma}\myp,
\end{equation}
where\/ $c_\gamma=g_\gamma(1)$ and the curvature
$\varkappa_\gamma(t)$ is given by\/ \textup{(\ref{eq:curvature1})}.
Then the functions\/ $\delta_1(t)$\textup{,} $\delta_2(t)$ satisfy
assumption\/ \textup{(\ref{ogd})} and equation\/
\textup{(\ref{eq:delta-eqn})}.
\end{proposition}
\begin{proof}
It is straightforward to verify that equation (\ref{eq:delta-eqn})
is satisfied. A lower bound of the form (\ref{ogd})
follows from assumption (\ref{kr}).
\end{proof}

\begin{remark}
In the ``classical'' case, where the curve $\gamma=\gamma^*$ is
determined by  equation (\ref{eq:gamma*}), it is easy to check that
the corresponding curvature (see (\ref{eq:curvature})) is given by
$$
\varkappa_{\gamma^*}(t)=\frac{c\mypp
(1+t/c)^3}{2\myp(1+t^2)^{3/2}},\qquad 0\le t\le\infty.
$$
Hence, expressions (\ref{eq:delta12}) are reduced to the constants
$\delta_1 =\kappa\myp(c/2)^{1/3}$, \,$\delta_2 =\delta_1/c$ \,(cf.\
\cite{BZ4}).
\end{remark}

%

\begin{assumption}\label{as:z}
Throughout the rest of the paper, we assume that the parameters
$z_1(x)$\textup{,} $z_2(x)$ $(x\in \mathcal{X})$ are chosen
according to formulas\/ \textup{(\ref{z})} with the functions\/
$\delta_1(t)$\textup{,} $\delta_2(t)$ given by\/
\textup{(\ref{eq:delta+infty})}\textup{,}
\textup{(\ref{eq:delta12})}. In particular, the measure
$\QQ^\gamma_z$ becomes dependent on $n=(n_1,n_2)$, as well as the
$\QQ^\gamma_z$-probabilities and the corresponding expected values.
\end{assumption}

\section{Asymptotics of the expectation}\label{sec5}

In this section, we derive a few corollaries from the above choice
of $z_1(x)$, $z_2(x)$, assuming throughout that Assumptions
\ref{as2} and \ref{as:z} are satisfied.

\begin{theorem}\label{rs}
The convergence in\/ \textup{(\ref{lim})} is uniform in\/
$t\in[0,\infty]$\textup{,}
\begin{align}\label{zrs}
\lim_{n\to\infty} \sup_{0\le t\le\infty}
|n_1^{-1}\EE_z^{\gamma}[\ell_{\varGamma}(t)]-\ell_{\gamma}(t)|=0.
\end{align}
\end{theorem}

We will use the following simple criterion
(see \cite[Lemma 4.3]{BZ4}).
\begin{lemma}\label{lm:8.1}
Let\/ $\{f_n(t)\}$ be a sequence of non-decreasing functions on a
finite interval\/ $[a,b\myp]$\textup{,} such that\textup{,} for
each\/ $t\in[a,b\myp]$\textup{,}
$\lim_{n\to\infty}f_n(t)=f(t)$\textup{,} where\/ $f(t)$ is a
continuous \textup{(}non-decreasing\textup{)} function on\/
$[a,b\myp]$. Then\/ $f_n(t)\to f(t)$ uniformly on\/ $[a,b\myp]$.
\end{lemma}

\begin{proof}[Proof of Theorem \textup{\ref{rs}}]
Note that for each $n$ the function
$$
f_n(t):=n_1^{-1}\EE_z^{\gamma}[\ell_{\varGamma}(t)]=
\frac{1}{n_1}\sum_{x\in \mathcal{X}(t)} |x| \EE_z^{\gamma}[\nu(x)]
$$
is non-decreasing in $t$ and the limiting function
$f(t):=\ell_{\gamma}(t)$ given by (\ref{ell}) is continuous on
$[0,\infty]$. Hence, by Lemma \ref{lm:8.1} the convergence in
(\ref{zrs}) is uniform in $t$ on every finite interval $[0,t^*]$. To
complete the proof, it suffices to check that for any
$\varepsilon>0$ and for large enough $n$, there exists $t^*<\infty$
such that for all $t\ge t^*$
\begin{equation}\label{e}
n_1^{-1}\EE_z^{\gamma}[\ell_{\varGamma}(\infty)-\ell_{\varGamma}(t)]
\le\varepsilon.
\end{equation}

Using (\ref{per+}), similarly to (\ref{eq:hkm}) we can write
\begin{align}
\notag
\EE_z^{\gamma}[\ell_{\varGamma}(\infty)-\ell_{\varGamma}(t)]&=
\sum_{k=1}^\infty\sum_{x\in \mathcal{X}\setminus \mathcal{X}(t)}
|x|\mypp \rme^{-\alpha_n k\langle \tilde x,\myp\delta(\tau(\tilde x))\rangle}\\
\label{eq:Ez} &\le
\sum_{k=1}^\infty\sum_{x_1=1}^\infty\sum_{x_2>tx_1} (x_1+x_2)\mypp
\rme^{-\alpha_n k\delta_{*}(x_1+x_2)/2}.
\end{align}
Note that the number of integer pairs $(x_1,x_2)$ (with $x_1\ge1$,
$x_2\ge0$) satisfying the conditions $x_1+x_2=y$ and $x_2>tx_1$ does
not exceed $y/(t+1)$. Hence, again using estimate (\ref{eq:hkm}), we
see that the right-hand side of (\ref{eq:Ez}) is bounded from above
by
\begin{align*}
\sum_{k=1}^\infty\sum_{y=1}^\infty \frac{y^2}{t+1}\,\rme^{-\alpha_n
k\delta_{*\myn} y/2} &=\frac{1}{t+1}\sum_{k=1}^\infty
\frac{O(1)}{(\alpha_n k)^{3}}
=\frac{O(1)}{\alpha_n^3\myp(t+1)}\mypp.
\end{align*}
Recalling that $\alpha_n^3=1/(\rho_n\myp n_1)\sim n_1^{-1}$, this
implies estimate (\ref{e}) for all $t$ large enough.
\end{proof}

\begin{theorem}\label{th:5.3}
Uniformly in\/ $t\in[0,\infty]$ we have
\begin{align}\label{lim_uv}
\lim_{n\to\infty}n_1^{-1}\EE_{z}^{\gamma}[\xi_1(t)]=
u_\gamma(t),\qquad
\lim_{n\to\infty}n_1^{-1}\EE_{z}^{\gamma}[\xi_2(t)]=
g_\gamma(u_\gamma(t)).
\end{align}
In particular\textup{,} for\/ $t=\infty$ this yields
\begin{equation}\label{lim_uv_infty}
\lim_{n\to\infty}n_1^{-1}\EE_{z}^{\gamma}(\xi_1)=1,\qquad
\lim_{n\to\infty}n_1^{-1}\EE_{z}^{\gamma}(\xi_2)=c_\gamma.
\end{equation}
\end{theorem}

\begin{proof}
Similarly to representation (\ref{per-mu}), one can show that
\begin{equation}\label{eq:Riemann1}
n_1^{-1}\EE_{z}^{\gamma}[\xi_1(t)]=\rho_n\sum_{k,\myp m=1}^\infty
m\mu(m) \sum_{x_1=1}^\infty\sum_{0\le x_2\le tx_1} \alpha_n^3
x_1\myp e^{-km\alpha_n \langle\tilde x,\myp\delta(\tau(\tilde x))
\rangle}.
\end{equation}
Assuming that $t_0\le t\le t_1$ and passing to the limit similarly
as in the proof of Theorem \ref{th:1}, we obtain, using
(\ref{eq:delta-eqn}) and making the substitution $x_2=sx_1$, that
\begin{align}
\notag
\lim_{n\to\infty} n_1^{-1}\EE_{z}^{\gamma}[\xi_1(t)]&=
\sum_{k,\myp m=1}^\infty m\mu(m) \iint_{0\le x_2\le tx_1}\!x_1\,
\rme^{-km \langle x,\myp\delta(\tau(x))}\,\dif{x}_1\dif{x}_2\\
\notag &=\sum_{k,\myp m=1}^\infty m\mu(m)
\frac{2}{(km)^3}\int_{t_0}^t\frac{\dif s}
{\bigl(\delta_1(s)+s\delta_2(s)\bigr)^3}\\
\notag
&=2\sum_{k=1}^\infty\frac{1}{k^3}\sum_{m=1}^\infty
\frac{\mu(m)}{m^2} \int_{t_0}^{t}\frac{\dif s}
{\kappa^3\myp g_\gamma''(u_\gamma(s))}\\
&=\frac{2\myp\zeta(3)}{\zeta(2)\mypp\kappa^3}
\int_{0}^{u_\gamma(t)}\frac{\dif
g_\gamma'(u)}{g_\gamma''(u)}=u_\gamma(t). \label{eq:Riemann2}
\end{align}
Similarly,
\begin{align*}
\lim_{n\to\infty} n_2^{-1}\EE_{z}^{\gamma}[\xi_2(t)] &=\sum_{k,\myp
m=1}^\infty m\mu(m) \iint_{0\le x_2\le tx_1}\!x_2\,
\rme^{-km \langle x,\myp\delta(\tau(x))} \,\dif{x}_1\dif{x}_2\\
&=\sum_{k,\myp m=1}^\infty m\mu(m)
\frac{2}{(km)^3}\int_{t_0}^{t}\frac{s\,\dif s}
{\bigl(\delta_1(s)+s\delta_2(s)\bigr)^3}\\
&=\frac{2\zeta(3)}{\zeta(2)}
\int_{t_0}^{u_\gamma(t)}\frac{s}
{\kappa^3\myp g_\gamma''(u_\gamma(s))}\,\dif s\\
&=\int_{0}^{u_\gamma(t)}\frac{g'_\gamma(u)\,\dif g_\gamma'(u)}
{g_\gamma''(u)}=g_\gamma(u_\gamma(t)).
\end{align*}
Finally, the uniform convergence in (\ref{lim_uv}) can be proved
similarly as in Theorem \ref{rs}.
\end{proof}

For the future applications, we need to estimate the rate of
convergence in (\ref{lim_uv_infty}) with sufficient accuracy. To
this end, we require some more smoothness of the function
$g_\gamma$.
\begin{assumption}\label{as3}
In addition to Assumptions \ref{as2} and \ref{as:z}, we now suppose
that $g_\gamma\in C^3[0,1]$.
\end{assumption}

\begin{theorem}\label{th:5.4}
Under Assumption \textup{\ref{as3}}\textup{,}
\,$\EE_z^\gamma(\xi_j)-n_j=O(n_1^{2/3})$ as\/ $n\to\infty$
\,\textup{(}$j=1,2$\textup{)}.
\end{theorem}
\begin{proof}
Consider $\xi_1$ (the case $\xi_2$ is handled similarly). From
(\ref{eq:Riemann1}) with $t=\infty$ we have
\begin{equation*}
\EE_{z}^{\gamma}(\xi_1)
=\sum_{k,\myp
m=1}^\infty \frac{\mu(m)}{k\alpha_n}\,F_1(km\alpha_n),
\end{equation*}
where
\begin{equation}\label{eq:Ff}
F_1(h):=\sum_{x_1=1}^\infty\sum_{x_2=0}^\infty f_1(hx_1,hx_2),\qquad
f_1(x_1,x_2):=x_1\myp \rme^{-\langle\tilde x,\myp\delta(\tau(\tilde
x))\rangle}.
\end{equation}
Repeating the calculations as in (\ref{eq:Riemann2}), we note that
$$
\iint_{\RR_+^2}
f_1(hx_1,hx_2)\,\dif{x}_1\dif{x}_2=\frac{2}{\rho_n\myp h^2\kappa^3},
$$
so that
\begin{equation}\label{eq:n1}
\begin{aligned}
\sum_{k,\myp m=1}^\infty \frac{\mu(m)}{\alpha_n k}
&\left(\left.\iint_{\RR_+^2}
f_1(hx_1,hx_2)\,\dif{x}_1\dif{x}_2\right)\right|_{h=\alpha_n km}
=\frac{2}{\rho_n\myp\alpha_n^3\kappa^3} \sum_{k,\myp m=1}^\infty
\frac{\mu(m)}{k^3m^2} =\frac{1}{\rho_n\myp\alpha_n^3}=n_1.
\end{aligned}
\end{equation}
Hence, we obtain the representation
\begin{equation}\label{eq:error1}
\EE_{z}^{\gamma}[\xi_1]-n_1= \sum_{k,\myp m=1}^\infty
\frac{\mu(m)}{\alpha_n k} \,\Delta_1(\alpha_n km),
\end{equation}
where
$$
 \Delta_1(h):=F_1(h)-
 \iint_{\RR_+^2} f_1(hx_1,hx_2)\,\dif{x}_1\dif{x}_2.
$$

Using that $\delta_i(t)\ge \delta_{*\myn}>0$ and $\rho_n\le 1/2$, we
have
\begin{align*}
  F_1(h)&\le \sum_{x_1=1}^\infty\sum_{x_2=0}^\infty hx_1
  \mypp \rme^{-h(x_1+x_2)\myp\delta_{*\myn}/2}
  =\frac{h\mypp \rme^{-h\delta_{*\myn}/2}}{(1-\rme^{-h\delta_{*\myn}/2})^3}\mypp.
\end{align*}
Hence, $F_1(h)=O(h^{-2})$ as $h\to0$ and $F_1(h)=O(h^{-\beta})$ for
any $\beta>0$ as $h\to+\infty$. Therefore, the function $F_1(h)$ is
well defined for all $h>0$ and its Mellin transform
\begin{equation}\label{eq:M}
M_1(s):=\int_0^\infty h^{s-1}F_1(h)\,\dif h
\end{equation}
(see, e.g., \cite[Ch.\,VI, \S\mypp{}9]{Widder})
is a regular function for $\Re s>2$. From a
two-dimensional version of the M\"untz formula (see
\cite[Lemma~5.1]{BZ4}),
it follows that
$M_1(s)$ is meromorphic in the
half-plane $\Re s>1$ and has the single $($\myn{}simple\myp$)$ pole
at point $s=2$. Moreover, for all $1<\Re s<2$
\begin{equation}\label{eq:Muntz}
M_1(s)=\int_0^\infty h^{s-1}\Delta_1(h)\,\dif h.
\end{equation}
The inversion formula for the Mellin transform
\cite[Theorem~9a, pp.\ 246--247]{Widder}
%
yields
\begin{equation}\label{eq:inversion}
\Delta_1(h)=\frac{1}{2\pi \rmi}\int_{c-\rmi\infty}^{c+\rmi\infty}
h^{-s} M_1(s)\,\dif s,\qquad 1<c<2.
\end{equation}

In order to make use of formula (\ref{eq:inversion}), we need to
find explicitly the analytic continuation of function (\ref{eq:M})
to the strip $1<\Re s<2$.
Let us use the Euler--Maclaurin summation formula
\begin{equation}\label{EM1}
\sum_{x=0}^\infty f(x)=\int_0^\infty\! f(x)\,\dif x +
\frac{1}{2}\,f(0)+ \int_0^\infty\! B_1(x)\,f'(x)\,\dif x,
\end{equation}
where $B_1(x):=x-[x]-1/2$ and $[x]$ is the integer part of $x$.
In view of Assumption \ref{as3} and equations (\ref{eq:curvature1}),
(\ref{eq:delta12}), we can apply this formula to the sum over $x_2$
in (\ref{eq:Ff}). Using the substitution $x_2=tx_1/\rho_n$, we
obtain
\begin{align}
\notag F_1(h)&=\sum_{x_1=1}^\infty h x_1 \int_0^\infty \rme^{-h
\langle\tilde x,\myp\delta(\tau(\tilde x)) \rangle}\,\dif x_2 +
\frac{1}{2}\sum_{x_1=1}^\infty hx_1\myp
\rme^{-hx_1\delta_1(0)}+ O(1)\,\frac{\rme^{-\const\cdot h}}{h}\\
&=\frac{h}{\rho_n}\sum_{x_1=1}^\infty x_1^2\int_0^\infty
\rme^{-hx_1\psi(t)}\,\dif t+O(1)\,\frac{\rme^{-\const\cdot
h}}{h}\myp, \label{eq:F}
\end{align}
where (cf.\ (\ref{eq:delta-eqn}))
\begin{equation}\label{eq:psi}
\psi(t):=\delta_1(t)+t\delta_2(t)\equiv \kappa\myp
g''_\gamma(u_\gamma(t))^{1/3}.
\end{equation}

Keeping track of only the main term in (\ref{eq:F}) and writing dots
for functions that are regular for $\Re s>1$, the Mellin transform
of $F_1(h)$ can be represented as follows
\begin{align}
\notag
M_1(s)&=\frac{1}{\rho_n}\int_0^\infty h^{s}
\left(\sum_{x_1=1}^\infty x_1^2\int_0^\infty
\rme^{-hx_1\psi(t)}\,\dif{t}\right)\dif{h}+\cdots\\
\notag
&=\frac{1}{\rho_n}\sum_{x_1=1}^\infty x_1^2 \int_0^\infty
\left( \int_0^\infty h^{s}\,\rme^{-hx_1\psi(t)}\,\dif{h}
\right)\dif{t}+\cdots\\
\notag
&=\frac{1}{\rho_n}\sum_{x_1=1}^\infty \frac{1}{x_1^{s-1}}
\int_0^\infty \frac{\Gamma(s+1)}{\psi(t)^{s+1}}\,\dif{t}+\cdots\\
&=\frac{1}{\rho_n}\,\zeta(s-1)\,\Gamma(s+1)\,\Psi(s)+\cdots,
\label{eq:M(s)}
\end{align} where
$$
\Psi(s):=\int_0^\infty \frac{1}{\psi(t)^{s+1}}\,\dif{t}.
$$
Recalling (\ref{eq:curvature}), function (\ref{eq:psi}) may be
rewritten in the form
$$
  \psi(t)=\kappa\mypp \varkappa_\gamma(t)^{1/3}\,\sqrt{1+t^2},\qquad
  t_0\le t\le t_1,
$$
and Assumption \ref{as2} implies that the function $\Psi(s)$ is
regular if $\Re s>0$. Furthermore, it is well known that the gamma
function $\Gamma(s)$ is analytic for $\Re s>0$ \cite[\S\mypp{}4.41,
p.\,148]{Titch2}, whereas the zeta function $\zeta(s)$ has a single
pole at point $s=1$ \cite[\S\mypp{}4.43, p.\,152]{Titch2}. It
follows that the right-hand side of (\ref{eq:M(s)}) is regular in
the strip $1<\Re s<2$ and hence provides the required analytic
continuation of the function $M_1(s)$ originally defined by
(\ref{eq:M}).

Setting $h=\alpha_n km$ and returning to formulas (\ref{eq:error1})
and (\ref{eq:inversion}), we get for $1<c<2$
\begin{align}
\notag
\EE_{z}^{\gamma}(\xi_1)-n_1&= \sum_{k,\myp m=1}^\infty
\frac{\mu(m)}{\alpha_n k} \,\frac{1}{2\pi \rmi}
\int_{c-\rmi\infty}^{c+\rmi\infty}\frac{M_1(s)}{(km\alpha_n)^{s}}\,\dif{s}\\
&=\frac{1}{2\pi \rmi}\int_{c-\rmi\infty}^{c+\rmi\infty}
\frac{M_1(s)\mypp\zeta(s+1)}{\alpha_n^{s+1}\mypp\zeta(s)}\,\dif{s}.
\label{eq:error1'}
\end{align} Using that $\zeta(s)\ne0$ for $\Re
s\ge1$, we can transform the contour of integration $\Re s=c$ in
(\ref{eq:error1'}) to the union of a small semi-circle $s=1+r\myp
\rme^{\myp\rmi t}$ \,($-\pi/2\le t\le\pi/2$) and two vertical lines,
$s=1\pm \rmi t$ \,($t\ge r$). Furthermore, studying resolution
(\ref{eq:M(s)}), one can show that $M_1(1\pm \rmi
t)=O\myp(|t|^{-2})$ as $t\to\infty$. As a result, the right-hand
side of (\ref{eq:error1'}) is bounded by $O(\alpha_n^{-2})$. Thus,
the proof of the theorem for $\xi_1$ is complete.
%
\end{proof}

\section{Asymptotics of higher-order moments}\label{sec6}
Throughout this section, we suppose that Assumptions \ref{as2} and
\ref{as:z} hold.

\subsection{Second-order moments}\label{sec6.1} Let
$K_z:=\Cov(\xi,\xi)=$ be the covariance matrix (with respect to the
measure $\QQ_z^\gamma$) of the random vector $\xi=\sum_{x\in\calX}
x\myp\nu(x)$. Since the random variables $\nu(x)$ are mutually
independent,  we see using (\ref{eq:nu}) that the elements
$K_z(i,j)=\Cov(\xi_i,\xi_j)$ ($i,j\in\{1,2\}$) of $K_z$ are given by
\begin{equation}\label{D_zxi}
K_z(i,j)=\sum_{x\in \calX}x_i x_j \Var[\nu(x)]=\sum_{x\in \calX} x_i
x_j \sum_{k=1}^\infty k\mypp z^{kx}.
\end{equation}


\begin{theorem}\label{th:2}
As $n\to\infty$\textup{,}
\begin{equation}\label{eq:Sigma}
K_z=3\myp\kappa^{-1}n_1^{4/3}\,(1+o(1))\myp B,
\end{equation}
where the elements of the matrix $B:=(B_{ij})$ are given by
\begin{equation}\label{eq:B}
B_{11}=\int_0^1 \frac{\dif u}{g''_\gamma(u)^{1/3}}\myp,\quad
B_{12}=B_{21}=\int_0^1 \frac{g'_\gamma(u)\,\dif
u}{g''_\gamma(u)^{1/3}}\myp,\quad B_{22}=\int_0^1
\frac{g'_\gamma(u)^2\, \dif u}{g''_\gamma(u)^{1/3}}\myp.
\end{equation}
\end{theorem}

\begin{proof}
Let us consider $K_z(1,1)$ (the other elements of $K_z$ are analyzed
in a similar manner). Substituting (\ref{z}) into (\ref{D_zxi}), by
the M\"obius inversion formula (cf.\ (\ref{eq:Riemann1})) we obtain
\begin{align}
\notag K_z(1,1)&=\sum_{k=1}^\infty\sum_{x\in \mathcal{X}} kx_1^2
\mypp\rme^{-k\alpha_n\langle \tilde x,\myp\delta(\tau(\tilde
x)\rangle}\\
&=\sum_{k,\myp m=1}^\infty km^2\mu(m)
\sum_{y_1=1}^\infty\sum_{y_2=0}^\infty y_1^2
\mypp\rme^{-km\alpha_n\langle\tilde y,\myp\delta(\tau(\tilde
y)\rangle}. \label{D_2}
\end{align}
Arguing as in the proof of Theorems \ref{th:1} and \ref{th:5.3}, we
obtain
\begin{align*}
 \lim_{n\to\infty} \alpha_n^4
 \sum_{x_1=1}^\infty\sum_{x_2=0}^\infty x_1^2
\mypp\rme^{-km\alpha_n\langle\tilde x,\myp\delta(\tau(\tilde
x)\rangle}&= \iint_{\RR_+^2}
x_2^2\mypp\rme^{-km\alpha_n\langle\tilde
x,\myp\delta(\tau(\tilde{x}))\rangle}\,\dif x_1\dif x_2\\
&=\frac{6}{(km)^4}
\int_{t_0}^{t_1}\frac{\dif{s}}{\bigl(\delta_1(s)+s\delta_2(s)\bigr)^4}\myp.
\end{align*}
Returning to (\ref{D_2}) and using (\ref{eq:delta+infty}),
(\ref{eq:delta-eqn}), we get
\begin{align*}
\lim_{n\to\infty}\alpha_n^4 K_z(1,1)
&=\frac{6\myp\zeta(3)}{\zeta(2)}\int_{t_0}^{t_1}
\frac{\dif{s}}{\kappa^4\myp
g''_\gamma(u_\gamma(s))^{4/3}}=\frac{3}{\kappa}\int_0^1
\frac{\dif{u}}{g''_\gamma(u)^{1/3}}\myp,
\end{align*}
and the first formula in (\ref{eq:B}) follows, since
$\alpha_n=(\rho_n \myp n_1)^{-1/3}$ and $\rho_n\to1$ as
$n\to\infty$.
%
%
%
\end{proof}

\begin{lemma}\label{lm:detK_z}
As\/ $n\to\infty$\textup{,}
\begin{equation}\label{eq:detK}
\det K_z\sim \left(\frac{3}{\kappa}\right)^2
\left(\int_0^1\frac{\dif{u}}{g''_\gamma(u)^{1/3}}
\int_0^1\frac{g'_\gamma(u)^2\,\dif u}{g''_\gamma(u)^{1/3}}-
\left(\int_0^1\frac{g'_\gamma(u)\,\dif{u}}{g''_\gamma(u)^{1/3}}\right)^2\right)
n_1^{8/3}.
\end{equation}
\end{lemma}

\begin{proof}
The proof readily follows from Theorem \ref{th:2}.
\end{proof}

From Theorem \ref{th:2} and Lemma \ref{lm:detK_z}, it follows (e.g.,
using the Cauchy--Schwarz inequality) that the matrix $K_z$ is
(asymptotically) positive definite; in particular, $\det K_z>0$ and
hence $K_z$ is invertible. Let $V_z=K_z^{-1/2}$ be the (unique)
square root of
$K_z^{-1}$,
that is, a symmetric, positive definite matrix such that
$V_z^2=K_z^{-1}$.
Recall that the matrix norm induced by the Euclidean vector
norm $|\,{\cdot}\,|$ is defined by $\|A\|:=\sup_{|x|=1}|x A|$.
We need some general facts about this norm (see \cite[\S\myp7.2,
pp.\ 33--34]{BZ4} for simple proofs and bibliographic comments).

\begin{lemma}\label{lm:6.2}
If\/ $A$ is a real matrix then $\|A^{\topp}\mynn
A\|=\|A\|^2$.
\end{lemma}

\begin{lemma}
\label{lm:6.3} If\/ $A=(a_{ij})$ is a\/ real $d\times d$
matrix\myp\textup{,} then\/
\begin{equation}\label{eq:m-norm}
\frac{1}{d}\sum_{i,\myp{}j=1}^d a_{ij}^2\le
\|A\|^2\le\sum_{i,\myp{}j=1}^d a_{ij}^2\myp.
\end{equation}
\end{lemma}

\begin{lemma}\label{lm:6.5}
Let\/ $A$ be a symmetric\/ $2\times 2$ matrix with\/ $\det A\ne 0$.
Then
\begin{equation}\label{eq:d=2}
\|A^{-1}\|=\frac{\|A\|}{|\mynn\det A|}\myp.
\end{equation}
\end{lemma}

We can now prove the following estimates for the norms of $K_z$ and
$V_z$.

\begin{lemma}\label{lm:K_z}
As\/ $n\to\infty$\textup{,} we have
\begin{equation}\label{eq:norms}
\|K_z\|\asymp n_1^{4/3},\qquad \|V_z\|\asymp n_1^{-2/3}.
\end{equation}
\end{lemma}

\begin{proof}
Using Theorem \ref{th:2} and the upper bound in Lemma \ref{lm:6.3},
we get
\begin{equation}\label{dL_3}
\|K_z\|^2\le K_z(1,1)^2+2 K_z(1,2)+ K_z(2,2)^2=O(n_1^{8/3}).
\end{equation}
On the other hand, by Theorem \ref{th:2} and the lower bound in
Lemma \ref{lm:6.3}
\begin{align}
\notag
\|K_n\|^2&\ge\frac{1}{2}\,\bigl(K_z(1,1)^2+K_z(2,2)^2\bigr)\\
\label{dL_4} &\ge K_z(1,1)\mypp
K_z(2.2)\sim\left(\frac{3}{\kappa}\right)^2 n_1^{8/3}
\int_0^1\frac{\dif u}{g''_\gamma(u)^{1/3}}
\int_0^1\frac{g'_\gamma(u)^2\,\dif u}{g''_\gamma(u)^{1/3}}\myp.
\end{align}
Combining (\ref{dL_3}) and (\ref{dL_3}) we obtain the first estimate
in (\ref{eq:norms}).

Further, Lemma \ref{lm:6.2} implies that $\|V_z\|^2=\|K_z^{-1}\|$.
In turn, Lemma \ref{lm:6.5} yields $\|K_z^{-1}\|=\|K_z\|/\det K_z$,
and it remains to use Lemmas \ref{lm:detK_z} and \ref{lm:K_z} to
obtain the second part of (\ref{eq:norms}).
\end{proof}

\subsection{Asymptotics of the moment sums}\label{sec6.3}

Denote $\nu_0(x):=\nu(x)-\EE_{z}^{\gamma}[\nu(x)]$ ($x\in \calX$),
and for $q\in\NN$ set
\begin{equation}\label{eq:m-mu}
m_q(x):=\EE_{z}^{\gamma}\myn\bigl[\nu(x)^q\bigr],\qquad
\mu_q(x):=\EE_{z}^{\gamma}\myn\bigl|\nu_0(x)^q\bigr|
\end{equation}
(for notational simplicity, we suppress the dependence on $\gamma$
and $z$).

The following two-sided estimate of $\mu_q(x)$ can be easily proved
using Newton's binomial formula and Lyapunov's inequality (cf.\
\cite[Lemmas 6.2 and 6.6]{BZ4}).
\begin{lemma}\label{lm:mu<m}
For each\/ $q\in\NN$ and all\/ $x\in \calX$\textup{,}
\begin{equation}\label{eq:mu<m}
\mu_2(x)^{q/2}\le \mu_q(x)\le 2^q \myp m_q(x).
\end{equation}
\end{lemma}

Next, we  need a general upper bound for the moments of geometric
random variables proved in \cite[Lemma~6.3]{BZ4}.
\begin{lemma}\label{lm:6.8}
For each\/ $q\in\NN$\textup{,} there exists a constant\/ $C_q>0$
such that\myp\textup{,} for all\/ $x\in \calX$\textup{,}
\begin{equation}\label{t=0}
  m_q(x)\le \frac{C_{q}z^{x}}{(1-z^x)^q}\myp.
\end{equation}
\end{lemma}

Using estimate (\ref{t=0}) and repeating the calculations in the
proof of Lemma 6.4 in \cite{BZ4}, one obtains the following
asymptotic bound.

\begin{lemma}\label{lm:6.9}
For each\/ $q\in\NN$\textup{,}
\begin{equation*}
\sum_{x\in \calX} |x|^q\mypp  m_q(x)=O(1)\,n_1^{(q+2)/3}, \qquad
n\to\infty.
\end{equation*}
\end{lemma}

Lemma \ref{lm:6.9}, together with bounds (\ref{eq:mu<m}) and Theorem
\ref{th:2}, implies the following asymptotic estimate (cf.\
\cite[Lemma~6.6]{BZ4}).
\begin{lemma}\label{lm:muk}
For any integer\/ $q\ge 2$\textup{,}
\begin{equation*}
\sum_{x\in \calX}|x|^q\mu_q(x)\asymp n_1^{(q+2)/3},\qquad
n\to\infty.
\end{equation*}
\end{lemma}

Using Lemma \ref{lm:muk} and the lower bound $\delta_j(t)\ge
\delta^*$ (see (\ref{ogd})), the next asymptotic bound is obtained
by a straightforward adaptation of the proof of a similar result in
\cite[Lemma~6.7]{BZ4}.

\begin{lemma}\label{lm:6.6}
For each\/ $q\in\NN$\textup{,}
\begin{equation*}
\EE_{z}^{\gamma}|\ell_\varGamma-\EE_{z}^{\gamma}(\ell_\varGamma)|^{q}
=O\bigl(n_1^{2q/3}\bigr),\qquad n\to\infty.
\end{equation*}
\end{lemma}

Finally, let us consider the \textit{Lyapunov coefficient}
\begin{equation}\label{L3}
L_z:=\|V_z\|^{3} \sum_{x\in \calX}|x|^3\mu_3(x),
\end{equation}
The next asymptotic estimate is an immediate consequence of Lemmas
\ref{lm:K_z} and \ref{lm:muk}.

\begin{lemma}\label{lm:7.1}
As\/ $n\to\infty$\textup{,} one has\/ $L_z\asymp n_1^{-1/3}$.
\end{lemma}

\section{Local limit theorem}\label{sec7}

The role of a local limit theorem in our approach is to yield the
asymptotics of the probability
$\QQ^\gamma_z\{\xi=n\}\equiv\QQ^\gamma_z(\CP_n)$ appearing in the
representation of the measure $\PP^\gamma_n$ as a conditional
distribution, $\PP^\gamma_n(\cdot)=\QQ^\gamma_z(\cdot\mypp|\CP_n)
=\QQ^\gamma_z(\cdot)/\QQ^\gamma_z(\CP_n)$.

As before, we denote $a_z:=\EE^\gamma_{z}(\xi)$ and $K_z:=
\Cov(\xi,\xi)= \EE^\gamma_{z}$.
Let $f_{0,I}(\cdot)$ be the density of a standard two-dimensional
normal distribution $\mathcal{N}(0,I)$ (i.e., with zero mean and
identity covariance matrix),
\begin{equation*}
f_{0,I}(x)=\frac{1}{2\pi}\,\rme^{-|x|^2\myn/2},\qquad x\in\RR^2.
\end{equation*}
Then the density of the normal distribution $\mathcal{N}(a_z,K_z)$
is given by
\begin{equation}\label{eq:phi1}
f_{a_z\myn,K_z}(x)= (\det K_z)^{-1/2}f_{0,I}\bigl((x-a_z)\mypp
V_z\bigr),\qquad x\in\RR^2.
\end{equation}


%

\begin{theorem}[Local limit theorem]\label{llt}
Under Assumptions \textup{\ref{as2}} and
\textup{\ref{as:z}}\textup{,} uniformly in $m\in\ZZ_+^2$
\begin{equation}\label{eq:LCLT}
\QQ_z^{\gamma}\{\xi=m\}=f_{a_z\myn,K_z}(m)+O(n_1^{-5/3}).
\end{equation}
\end{theorem}

Let us make some preparations for the proof. Recall that the random
variables $\{\nu(x)\}_{x\in \mathcal{X}}$
are mutually independent and have
geometric distribution with parameter $z^{x}$, respectively.
In particular, the characteristic function
$\varphi_{\nu}(t):=\EE^\gamma_{z} (\rme^{\myp\rmi t\nu})$ of
$\nu(x)$ is given by
\begin{equation}\label{f_nu}
\varphi_{\nu}(t;x)=\frac{1-z^{x}}{1-z^{x} \mypp\rme^{\myp\rmi
t}}\myp;
\end{equation}
hence, the characteristic function $\varphi_{\xi}(\lambda):=\EE_{z}
(\rme^{\myp\rmi\langle\lambda,\,\xi\rangle})$ of the vector
$\xi=\sum_{x\in \mathcal{X}} x\nu(x)$ reads
\begin{equation}\label{f_xi}
\varphi_{\xi}(\lambda)=\prod_{x\in \mathcal{X}}
\varphi_{\nu}(\langle x,\lambda\rangle; x)=\prod_{x\in \mathcal{X}}
\frac{1-z^{x}} {1-z^{x}\mypp \rme^{\myp\rmi\langle
x,\lambda\rangle}}\mypp.
\end{equation}

Let us start with a general absolute estimate for the characteristic
function of a centered random variable (for a proof, see \cite[Lemma
7.10]{BZ4}).
\begin{lemma}\label{lm:7.2_f}
Let\/ $\varphi_{\nu_0}(t;x):=\EE^\gamma_{z} (\rme^{\myp\rmi\myp
t\myp \nu_0(x)})$ be the characteristic function of the random
variable $\nu_0(x):=\nu(x)-\EE^\gamma_{z}[\nu(x)]$. Then
\begin{align}\label{x.f_4}
|\varphi_{\nu_0}(t;x)|\le
\exp\Bigl\{-{\textstyle\frac12}\myp\mu_2(x)t^2+
{\textstyle\frac13}\myp\mu_3(x)|t|^3\Bigr\},\qquad t\in\RR\myp.
\end{align}
\end{lemma}

The next lemma provides two estimates (proved in \cite[Lemmas 7.11
and 7.12]{BZ4})
for the characteristic function
$\varphi_{\xi_0}(\lambda):=\EE^\gamma_{z}(\rme^{\myp\rmi\langle\lambda,\,\xi_0\rangle})$
of the centered vector $\xi_0:=\xi-a_z
=\sum_{x\in\calX}x\myp\nu_0(x)$. Recall that the Lyapunov
coefficient $L_z$ is defined in \textup{(\ref{L3})}, and
$V_z:=K_z^{-1/2}$.
\begin{lemma}\label{lm:7.2_F}
\textup{(a)} \,For all\/ $\lambda\in\RR^2$\textup{,}
\begin{equation}\label{x.f_6}
|\varphi_{\xi_0}(\lambda V_z)|\le \exp\bigl\{-{\textstyle
\frac12}\myp|\lambda|^2+{\textstyle\frac{1}{3}}\myp
L_z|\lambda|^3\bigr\}.
\end{equation}

\textup{(b)} \,If\/ $|\lambda|\le L_z^{-1}$ then
\begin{equation}\label{16L}
\left|\varphi_{\xi_0}(\lambda V_z)-\rme^{-|\lambda|^2\myn/2}\right|
\le 16\myp L_z|\lambda|^3\mypp \rme^{-|\lambda|^2/6}.
\end{equation}
\end{lemma}

The next global bound is obtained by repeating the proof of
Lemma~7.14 in \cite{BZ4}.
\begin{lemma}\label{lm:7.3}
For all\/ $\lambda\in\RR^2$\textup{,}
\begin{equation}\label{f_J}
|\varphi_{\xi_0}(\lambda)|\le \rme^{-J_{n}(\lambda)},
\end{equation}
where
\begin{equation}\label{J_0}
J_{n}(\lambda):=\frac14\sum_{x\in \calX}
\rme^{-\alpha_n\langle\delta,\myp
x\rangle}\bigl(1-\cos\langle\lambda,x\rangle\bigr)\ge0.
\end{equation}
\end{lemma}

We can now proceed to the proof of Theorem \ref{llt}.

\begin{proof}[Proof of Theorem \textup{\ref{llt}}]
By the Fourier inversion formula, we can write
\begin{equation}\label{l_1}
\QQ^\gamma_{z}\{\xi=m\}=\frac{1}{4\pi^2}\int_{T^2}
\rme^{-\rmi\langle \lambda,\mypp{}m-a_z\myn\rangle}\mypp
\varphi_{\xi_0}(\lambda)\,\dif{}\lambda,\qquad m\in\ZZ_+^2\myp,
\end{equation}
where
$T^2:=\{\lambda=(\lambda_1,\lambda_2)\in\RR^2:|\lambda_1|\le\pi,\,
|\lambda_2|\le\pi\}$. On the other hand, the characteristic function
corresponding to the normal probability density $f_{a_z\myn,K_z}(x)$
(see (\ref{eq:phi1})) is given by
\begin{equation*}
\varphi_{a_z\myn,K_z}(\lambda)=
\rme^{\myp\rmi\langle\lambda,\myp{}a_z\myn\rangle-|\lambda
V_z^{-1}\myn|{\vphantom{(_z}}^2\myn/2},\qquad \lambda\in\RR^2,
\end{equation*}
so by the Fourier inversion formula
\begin{equation}\label{f_o}
f_{a_z\myn,K_z}(m)= \frac{1}{4\pi^2}
\int_{\RR^2}\rme^{-\rmi\langle\lambda,\mypp{}m-a_z\myn\rangle-|\lambda
V_z^{-1}\myn|{\vphantom{(_z}}^2\myn/2}\,\dif{}\lambda\myp,\qquad
m\in\ZZ^2_+\myp.
\end{equation}

Note that if $|\lambda V_z^{-1}\myn|\le L_z^{-1}$ then, according to
Lemmas \ref{lm:K_z} and \ref{lm:7.1},
\begin{equation*}
|\lambda|\le |\lambda V_z^{-1}\myn|\cdot\|V_z\| \le L_z^{-1}
\|V_z\|=O\bigl(n_1^{-1/3}\bigr)=o(1),
\end{equation*}
which of course implies that $\lambda\in T^2$. Using this
observation and subtracting (\ref{f_o}) from (\ref{l_1}), we get,
uniformly in $m\in\ZZ^2_+$\myp,
\begin{equation}\label{I}
\bigl|\QQ^\gamma_{z}\{\xi=m\}-f_{a_z\myn,K_z}(m)\bigr|\le
I_1+I_2+I_3\myp,
\end{equation}
where
\begin{align*}
&I_1:=\frac{1}{4\pi^2}\int_{\{\lambda\,:\,|\lambda V_z^{-1}\myn|
      \le L_z^{-1}\}}
      \bigl|\varphi_{\xi_0}(\lambda)-\rme^{-|\lambda V_z^{-1}\myn|{\vphantom{(_z}}^2\myn/2}
      \bigr|\,\dif{}\lambda\myp,\\
&I_2:=\frac{1}{4\pi^2}\int_{\{\lambda\,:\,|\lambda V_z^{-1}\myn|>
      L_z^{-1}\}}
      \rme^{-|\lambda V_z^{-1}\myn|{\vphantom{(_z}}^2\myn/2}\,\dif{}\lambda\myp,\\
&I_3:=\frac{1}{4\pi^2}
      \int_{T^2\cap\{\lambda\,:\,|\lambda V_z^{-1}\myn|>L_z^{-1}\}}
      |\varphi_{\xi_0}(\lambda)|\:\dif{}\lambda\myp.
\end{align*}

By the substitution $\lambda=y\myp V_z$, the integral $I_1$ is
reduced to
\begin{align}
\notag
I_1&=\frac{|\mynn\det V_z|}{4\pi^2} \int_{|y|\le L_z^{-1}}
\bigl|\varphi_{\xi_0}(y V_z)-
\rme^{-|y|^2\myn/2}\bigr|\,\dif{}y\\
&=O(1)\mypp(\det K_z)^{-1/2}\myp L_z\int_{\RR^2}
|y|^3\rme^{-|y|^2\myn/6}\,\dif{}y= O(n_1^{-5/3}),
\label{I1}
\end{align}
on account of Lemmas \ref{lm:detK_z},
\ref{lm:7.1} and \ref{lm:7.2_F}(b). Similarly, again putting
$\lambda=y\myp V_z$ and passing to the polar coordinates, we get,
due to Lemmas \ref{lm:detK_z} and \ref{lm:7.1},
\begin{equation}\label{I2}
\begin{aligned}
I_2&= \frac{|\mynn\det V_z|}{2\pi} \int_{L_z^{-1}}^\infty r\mypp
\rme^{-r^2\myn/2}\,\dif{r}=O(n_1^{-4/3})\,
\rme^{-L_z^{-2}\myn/2}=o(n_1^{-5/3}).
\end{aligned}
\end{equation}

Finally, let us turn to $I_3$. Using Lemma \ref{lm:7.3}, we obtain
\begin{align}\label{I3}
I_3&= O(1)\int_{T^2\cap \{|\lambda V_z^{-1}\myn|>L_z^{-1}\}}
\rme^{-J_{n}(\lambda)}\,\dif{}\lambda\myp,
\end{align}
where $J_{n}(\lambda)$ is given by (\ref{J_0}). The condition
$|\lambda V_z^{-1}\myn|>L_z^{-1}$ implies that
$|\lambda|>\sqrt{2}\,\eta\myp\alpha_n$ and hence
$\max\{|\lambda_1|,|\lambda_2|\}>\eta\myp\alpha_n$, where $\eta>0$
is suitable (small enough) constant. Indeed, assuming the contrary,
from (\ref{eq:alpha}) and Lemmas \ref{lm:K_z} and \ref{lm:7.1} it
would follow
\begin{align*}
1<L_z|\lambda V_z^{-1}\myn|&\le L_z\myp\eta\myp\alpha_n
\|K_z\|^{1/2}= O(\eta)\to0\quad\text{as}\ \ \eta\downarrow0,
\end{align*}
which is a contradiction. Hence, estimate (\ref{I3}) is reduced to
\begin{align}\label{I3.1}
I_3&=O(1) \left(
\int_{|\lambda_1|>\eta\myp\alpha_n}+\int_{|\lambda_2|>\eta\myp\alpha_n}\right)
\rme^{-J_{n}(\lambda)}\,\dif{}\lambda.
\end{align}

Note that, by Assumption \ref{as2} and formulas (\ref{eq:delta12}),
the functions $\delta_1(t)$, $\delta_2(t)$ are bounded above,
$\sup_t\delta_j(t)\le \delta^{*\mynn}<\infty$. Hence, (\ref{J_0})
implies
\begin{equation}\label{J1}
J_{n}(\lambda) \ge\sum_{x\in X} e^{-\alpha_n\delta^{*\myn}(x_1+x_2)}
\bigl(1-\cos\langle\lambda,x\rangle\bigr).
\end{equation}

To estimate the first integral in (\ref{I3.1}), by keeping in
summation (\ref{J1}) only pairs of the form $x=(x_1,1)$,
\,$x_1\in\ZZ_+$\myp, we obtain
\begin{align}
\notag J_{n}(\lambda)\,\rme^{\alpha_n\delta^*} \ge
\sum_{x_1=0}^\infty \rme^{-\alpha_n\delta^* x_1}\!\left(1-\Re\,
\rme^{\myp\rmi (\lambda_1 x_1+\lambda_2)}\right)
&=\frac{1}{1-\rme^{-\alpha_n}}-
\Re\left(\frac{\rme^{\myp\rmi\lambda_2}}{1-\rme^{- \alpha_n+\rmi\lambda_1}}\right)\\
\label{J_1}& \ge
\frac{1}{1-\rme^{-\alpha_n}}-\frac{1}{|1-\rme^{-\alpha_n +
\rmi\lambda_1}|}\myp,
\end{align}
because $\Re\myp u\le |u|$ for any $u\in\CC$. Since
$\eta\mypp\alpha_n\le|\lambda_1|\le \pi$, we have
\begin{align*}
|1-\rme^{-\alpha_n+\rmi\lambda_1}|
&\ge|1-\rme^{-\alpha_n+\rmi\myp\eta\myp\alpha_n}|\sim\alpha_n
(1+\eta^2)^{1/2} \qquad (\alpha_n\to0).
\end{align*}
Substituting this estimate into (\ref{J_1}), we conclude that
$J_n(\lambda)$ is asymptotically bounded from below by
$C(\eta)\mypp\alpha_n^{-1}\mynn \asymp n_1^{1/3}$ (with some
constant $C(\eta)>0$), uniformly in $\lambda$ such that
$\eta\mypp\alpha_n\le|\lambda_1|\le \pi$. Thus, the first integral
in (\ref{I3.1}) is bounded by
$O(1)\exp\bigl(-\myp\const \cdot n_1^{1/3}\bigr)=o(n_1^{-5/3})$.

Similarly, the second integral in (\ref{I3.1}) is estimated by
reducing the summation in (\ref{J_0}) to that over $x=(1,x_2)$ only.
As a result, $I_3=o(n_1^{-5/3})$. Substituting this estimate,
together with (\ref{I1}) and (\ref{I2}), into (\ref{I}) we get
(\ref{eq:LCLT}), and so the theorem is proved.
\end{proof}

\begin{corollary}\label{cor:Q}
In addition to the conditions of Theorem
\textup{\ref{llt}}\textup{,} suppose that Assumption
\textup{\ref{as3}} holds. Then\textup{,} as $n\to\infty$,
\begin{equation}\label{sim}
\QQ_z^\gamma\{\xi=n\}\asymp n_1^{-4/3}.
\end{equation}
\end{corollary}
\begin{proof}
By Theorem~\ref{th:5.4}, $a_z=\EE_z^\gamma(\xi)=n+O(n_1^{2/3})$.
Together with Lemma \ref{lm:K_z} this implies
$|(n-a_z)\myp V_z|\le |n-a_z|\,{\cdot}\, \|V_z\| =O(1)$.
Hence, by Lemma \ref{lm:detK_z} we get
\begin{align*}
f_{a_z\myn,K_z}(n)&=\frac{1}{2\pi}\,(\det K_z)^{-1/2}
\,\rme^{-|(n-a_z\myn)V_z|^2\myn/2}\asymp n_1^{-4/3},
\end{align*}
and (\ref{sim}) now readily follows from (\ref{eq:LCLT}).
\end{proof}

\section{Limit shape}\label{sec8}
Throughout this section we work under Assumptions
\textup{\ref{as2}}\textup{,} \textup{\ref{as:z}} and
\textup{\ref{as3}}. Let us first establish that a given curve
$\gamma\in\mathfrak{G}$ is indeed the limit shape of polygonal lines
$\varGamma\in\CP$ with respect to the measure $\QQ_z^{\gamma}$
(under the scaling $\varGamma\mapsto n_1^{-1}\varGamma$).

\begin{theorem}\label{zbc1}
For any $\varepsilon>0$\textup{,}
$$
\lim_{n\to\infty} \QQ_z^{\gamma}\left\{\varGamma\in\CP:
\,d_{\mathcal{L}}(n_1^{-1}\varGamma,\gamma)\le\varepsilon\right\}=1.
$$
\end{theorem}

\begin{proof}
In view of Theorem \ref{rs}, we only need to check that for each
$\varepsilon>0$
\begin{align}\label{>e}
\lim_{n\to\infty}\QQ_z^{\gamma}\left\{\frac{1}{n_1}\sup_{0\le
t\le\infty}\bigl|
\ell_{\varGamma}(t)-\EE_z^{\gamma}[\ell_{\varGamma}(t)]\bigr|>
\varepsilon\right\}=0.
\end{align}
Note that the random process
\begin{equation}\label{eq:tilde-l}
\ell^{\myp0}_{\varGamma}(t):=\ell_{\varGamma}(t)-\EE_z^\gamma[\ell_{\varGamma}(t)]\qquad(0\le
t\le\infty)
\end{equation}
has independent increments and zero mean, hence it is a martingale
with respect to the filtration ${\calF}_t:=\sigma\{\nu(x),\, x\in
\mathcal{X}(t), \, t\in[0,\infty]\}$. From the definition of
$\ell_\varGamma(t)$ (see (\ref{ell_Gamma})), it is also clear that
$\ell^{\myp0}_{\varGamma}(t)$ is c\`adl\`ag (i.e, its paths are
everywhere right-continuous and have left limits). Therefore,
Kolmogorov--Doob's submartingale inequality (see, e.g.,
\cite[Corollary 2.1]{Yeh})
gives
\begin{equation}\label{eq:KD+2}
\QQ_z^{\gamma}\left\{\sup_{0\le t\le\infty}
|\ell^{\myp0}_{\varGamma}(t)|> n_1\varepsilon\right\}\le
\frac{1}{(n_1\varepsilon)^2} \sup_{0\le t\le\infty}
\Var[\ell_{\varGamma}(t)]\le \frac{1}{n_1^2\myp\varepsilon^2}
\Var(\ell_{\varGamma}).
\end{equation}
Furthermore, using decomposition (\ref{ell_Gamma}) and Theorem
\ref{th:2}, we have
\begin{align}
\notag
\Var(\ell_{\varGamma})=\sum_{x\in\calX}|x|\mypp\Var[\nu(x)]&\le
\sum_{x\in\calX}(x_1+x_2)\mypp\Var[\nu(x)]\\
\label{eq:<n^4/3}&=\Var(\xi_1)+\Var(\xi_2)=O(n_1^{4/3}).
\end{align}
Finally, substituting (\ref{eq:<n^4/3}) into (\ref{eq:KD+2}), we see
that the probability on the left-hand side is bounded by
$O(n_1^{-2/3})\to0$, which proves (\ref{>e}).
\end{proof}

Let us now prove a limit shape result under the measure
$\PP_n^{\gamma}$ (cf.\ Theorem \ref{th:main1}).

\begin{theorem}\label{zbc2}
For any\/ $\varepsilon>0$
\begin{align}\label{<e}
\lim_{n\to\infty}\PP_{n}^{\gamma}
\left\{\varGamma\in\CP_n:,d_{\mathcal{L}}(n_1^{-1}
\varGamma,\gamma)\le\varepsilon\right\}=1.
\end{align}
\end{theorem}

\begin{proof}
Similarly to the proof of Theorem \ref{zbc1}, it suffices to show
that for each $\varepsilon>0$
\begin{align*}
\lim_{n\to\infty}\PP_n^{\gamma}\left\{\sup_{0\le t\le\infty}
|n_1^{-1}\ell^{\myp0}_{\varGamma}(t)|>\varepsilon\right\}=0,
\end{align*}
where the random process $\ell^{\myp0}_\varGamma(t)$ is defined in
(\ref{eq:tilde-l}). Recalling formula (\ref{u_1}), we obtain
\begin{align}\label{c/z}
\PP_n^{\gamma}\left\{\sup_{0\le
t\le\infty}|\ell^{\myp0}_{\varGamma}(t)|>\varepsilon n_1\right\}
&\le \frac{\displaystyle \QQ_z^{\gamma}\left\{\sup_{0\le
t\le\infty}|\ell^{\myp0}_{\varGamma}(t)|>\varepsilon n_1\right\}}
{\displaystyle \QQ_z^{\gamma}\{\xi=n\}}\myp.
\end{align}
To estimate the numerator in (\ref{c/z}), similarly to the proof of
Theorem \ref{zbc1} we use Kolmogorov--Doob's submartingale
inequality, but now with the sixth order central moment. Combining
this with Lemma \ref{lm:6.6} (with $q=3$), we obtain
\begin{align}\label{P<e^6}
\QQ_z^{\gamma}\left\{\sup_{0\le t\le\infty}
|\ell^{\myp0}_{\varGamma}(t)|> n_1\myp\varepsilon\right\}
&\le\frac{1}{n_1^6\myp\varepsilon^6}\,
\EE_z^{\gamma}\bigl|\ell_{\varGamma}-\EE_z^{\gamma}
(\ell_{\varGamma})\bigr|^6=O(n_1^{-2}).
\end{align}
On the other hand, by Corollary \ref{cor:Q} the denominator in
(\ref{c/z}) decays no faster than at order $n^{-4/3}$. Together with
(\ref{P<e^6}), this implies that the right-hand side of (\ref{c/z})
admits an asymptotic bound $O(n_1^{-2/3})\to0$. Hence, Theorem
\ref{zbc2} is proved.
\end{proof}

\section*{Acknowledgements}

The authors wish to thank Ya.\,G.\ Sinai for stimulating
discussions.


\begin{thebibliography}{99}

\small
\bibitem{ABT}
\textsc{Arratia, R., Barbour, A.D.\ and Tavar\'e, S.} (2003).
\textit{Logarithmic Combinatorial Structures\textup{:} A
Probabilistic Approach}.
European Math.\ Soc., Z\"urich. \MR{2032426}

\bibitem{Barany}
\textsc{B\'ar\'any, I.} (1995). The limit shape of convex lattice
polygons. \textit{Discrete Comput.\ Geom.} {\bf 13}, 279--295.
\MR{1318778}


\bibitem{BGT}
\textsc{Bingham, N.H., Goldie, C.M.\ and Teugels, J.L.} (1989).
\textit{Regular Variation}.
Cambridge Univ.\ Press, Cambridge.



\bibitem{Bogachev}
\textsc{Bogachev, L.V.} (2011). Limit shape of random convex
polygonal lines on $\ZZ^2$: Even more universality. Manuscript in
preparation.

\bibitem{BZ2}
\textsc{Bogachev, L.V.\ and Zarbaliev, S.M.} (1999). Approximation
of convex functions by random polygonal lines. \textit{Doklady
Math.} {\bf 59}, 46--49. \MR{1706217}


\bibitem{BZ3}
\textsc{Bogachev, L.V.\ and Zarbaliev, S.M.} (2004). Approximation
of convex curves by random lattice polygons. Preprint NI04003-IGS,
Isaac Newton Inst.\ Math.\ Sci., Cambridge.
\url{http://www.newton.cam.ac.uk/preprints/NI04003.pdf}

\bibitem{BZ4}
\textsc{Bogachev, L.V.\ and Zarbaliev, S.M.} (2008). Universality of
the limit shape of convex lattice polygonal lines. Preprint,
\url{http://arxiv.org/abs/0807.3682}. To appear in \textit{Annals of
Probability.}

\bibitem{BZ-DAN}
\textsc{Bogachev, L.V.\ and Zarbaliev, S.M.} (2009). A proof of the
Vershik--Prokhorov conjecture on the universality of the limit shape
for a class of random polygonal lines.
\textit{Doklady Math.} \textbf{79},
197--202. \MR{2541116}

\bibitem{Feller}
\textsc{Feller, W.} (1968) \textit{An Introduction to Probability
Theory and Its Applications\textup{,} Vol.\,I}, 3rd ed.
Wiley, New York. \MR{0228020}

\bibitem{HW}
\textsc{Hardy, G.H.\ and Wright, E.M.} (1960) \textit{An
Introduction to the Theory of Numbers}, 4th ed. Oxford Univ.\ Press,
Oxford.






\bibitem{Prokhorov}
\textsc{Prokhorov, Yu.V.} (1998). Private communication.

\bibitem{Sinai}
\textsc{Sinai, Ya.G.} (1994). A probabilistic approach to the
analysis of the statistics of convex polygonal lines.
\textit{Funct.\ Anal.\ Appl.} {\bf 28}, 108--113. \MR{1283251}

\bibitem{Titch2}
\textsc{Titchmarsh, E.C.} (1939). \textit{The Theory of Functions},
2nd ed. Oxford Univ.\ Press, Oxford.



\bibitem{V1}
\textsc{Vershik, A.M.} (1994). The limit form of convex integral
polygons and related problems. \textit{Funct.\ Anal.\ Appl.} {\bf
28}, 13--20. \MR{1275724}


\bibitem{V3}
\textsc{Vershik, A.M.} (1996) Statistical mechanics of combinatorial
partitions, and their limit configurations. \textit{Funct.\ Anal.\
Appl.} {\bf 30}, 90--105. \MR{1402079}



\bibitem{Widder}
\textsc{Widder, D.V.} (1946). \textit{The Laplace Transform}.
Princeton Univ.\ Press, Princeton, NJ. \MR{0005923}

\bibitem{Yeh}
\textsc{Yeh, J.} (1995). \textit{Martingales and Stochastic
Analysis.} World Scientific, Singapore. \MR{1412800}


\end{thebibliography}
\end{document}